\newcommand{\R}{\mathds R}
\newcommand{\C}{\mathds C}
\newcommand{\Z}{\mathds Z}
\newcommand{\N}{\mathds N}

\newcommand{\llangle}{\langle\!\!\langle}
\newcommand{\rrangle}{\rangle\!\!\rangle}
\newcommand{\Dds}{\tfrac{\mathrm D}{\mathrm ds}}
\newcommand{\dds}{\tfrac{\mathrm d}{\mathrm ds}}

\newcommand{\Dim}{\mathrm{dim}}
\newcommand{\Ker}{\mathrm{Ker}}
\newcommand{\de}{\mathrm{d}}
\newcommand{\Hi}{{\mathcal H}}

\documentclass[twoside,final,a4paper,10pt]{amsart}

\usepackage{times}
\usepackage{dsfont}

\renewcommand{\contentsline}[3]{\csname new#1\endcsname{#2}{#3}}
\newcommand{\newchapter}[2]{\bigskip\hbox to \hsize{\vbox{\advance\hsize by -.5cm\baselineskip=12pt\parfillskip=0pt\leftskip=2cm\noindent\hskip -2cm #1\leaders\hbox{.}\hfil\hfil\par}$\,$#2\hfil}}
\newcommand{\newsection}[2]{\medskip\hbox to \hsize{\vbox{\advance\hsize by -.5cm\baselineskip=12pt\parfillskip=0pt\leftskip=2.5cm\noindent\hskip -2cm #1\leaders\hbox{.}\hfil\hfil\par}$\,$#2\hfil}}
\newcommand{\newsubsection}[2]{\medskip\hbox to \hsize{\vbox{\advance\hsize by -.5cm\baselineskip=12pt\parfillskip=0pt\leftskip=3.5cm\noindent\hskip -2cm #1\leaders\hbox{.}\hfil\hfil\par}$\,$#2\hfil}}

\numberwithin{equation}{section}

\title[Iteration of closed Lorentzian geodesics]{Iteration of closed geodesics\\ in stationary
Lorentzian manifolds}
\author[M. A. Javaloyes]{Miguel Angel Javaloyes}
\address{Departamento de Matem\'atica,\hfill\break\indent
Universidade de S\~ao Paulo, \hfill\break\indent Rua do Mat\~ao
1010,\hfill\break\indent CEP 05508-900, S\~ao Paulo, SP, Brazil}
\email{majava@ime.usp.br}
\author[L. L. de Lima]{Levi Lopes de Lima}
\address{Departamento de Matem\'atica,\hfill\break\indent
Universidade Federal do Cear\'a, \hfill\break\indent
Campus do Pici,\hfill\break\indent
R. Humberto Monte, s/n,\hfill\break\indent CEP 60455-760, Fortaleza/CE, Brazil.}
\email{levi@mat.ufc.br}
\author[P.\ Piccione]{Paolo Piccione}
\address{Departamento de Matem\'atica,\hfill\break\indent
Universidade de S\~ao Paulo, \hfill\break\indent Rua do Mat\~ao
1010,\hfill\break\indent CEP 05508-900, S\~ao Paulo, SP, Brazil}
\email{piccione@ime.usp.br}
\urladdr{http://www.ime.usp.br/\~{}piccione}
\thanks{M. A. J. is sponsored by Fapesp; L. L. L. and P. P. are partially sponsored by CNPq.
The research project that lead to the results in this paper was developed while M. A. J. and P. P.
were visiting the Mathematics Department of the
Universidade Federal do Cear\'a, Fortaleza, Brazil, in January 2007.
The authors express their gratitude to the \emph{Instituto do Mil\^enio}, IMPA, Rio de Janeiro, Brazil,
for providing financial support to this project.}


\date{May 4th, 2007}

\begin{document}


\theoremstyle{plain}\newtheorem*{teon}{Theorem}
\theoremstyle{definition}\newtheorem*{defin*}{Definition}
\theoremstyle{plain}\newtheorem{teo}{Theorem}[section]
\theoremstyle{plain}\newtheorem{prop}[teo]{Proposition}
\theoremstyle{plain}\newtheorem{lem}[teo]{Lemma}
\theoremstyle{plain}\newtheorem{cor}[teo]{Corollary}
\theoremstyle{definition}\newtheorem{defin}[teo]{Definition}
\theoremstyle{remark}\newtheorem{rem}[teo]{Remark}
\theoremstyle{plain} \newtheorem{assum}[teo]{Assumption}
\swapnumbers
\theoremstyle{definition}\newtheorem{example}{Example}[section]
\theoremstyle{plain} \newtheorem*{acknowledgement}{Acknowledgements}
\theoremstyle{definition}\newtheorem*{notation}{Notation}


\begin{abstract}
Following the lines of \cite{Bo4}, we study the Morse index of the iterated of
a closed geodesic in stationary Lorentz\-ian manifolds, or, more generally,
of a closed Lorentz\-ian geodesic that admits a timelike periodic Jacobi field.
Given one such closed geodesic $\gamma$, we prove the existence of a locally constant integer valued map
$\Lambda_\gamma$ on the unit circle with the property that the Morse index of the iterated
$\gamma^N$ is equal, up to a correction term $\epsilon_\gamma\in\{0,1\}$, to the sum of
the values of $\Lambda_\gamma$ at the $N$-th roots of unity. The discontinuities of $\Lambda_\gamma$
occur at a finite number of points of the unit circle, that are special eigenvalues of the
linearized Poincar\'e map of $\gamma$.
We discuss some applications of the theory.
\end{abstract}

\maketitle
\tableofcontents

\begin{section}{Introduction}
It is well known that, unlike the Riemannian case, the geodesic action functional of
a manifold endowed with a non positive definite metric tensor is always \emph{strongly indefinite},
i.e., all its critical points have infinite Morse index. However, given a Lorentzian
manifold $(M,\mathfrak g)$ that has a timelike Killing vector field $\mathcal K$, one can consider a constrained geodesic
variational problem whose critical points have finite Morse index (\cite{BilMerPic, asian, CAG1, Masiello});
the value of this index is computed in terms of a symplectic invariant related to the Conley--Zehnder index
and the Maslov index of the linearized geodesic flow.

Following the classical Riemannian results, one wants to prove
multiplicity results for closed geodesics using variational methods,
including equivariant Morse theory. The closed geodesic variational problem is
invariant by the action of the compact Lie group $\mathrm O(2)$;
every $\mathrm O(2)$-orbit of a closed geodesic contributes to the
homology of the free loop space.
In order to obtain multiplicity results, one needs to distinguish between
critical orbits generated by the \emph{tower of iterates} $(\gamma^N)_{N\ge1}$
of the same geodesic $\gamma$. As proved by Gromoll and Meyer in the celebrated
paper \cite{GroMey2}, fine estimates on the homological contribution of
iterated closed geodesics can be given in terms of the Morse index and the nullity
of the iterate. Thus, an essential step in the development of the Morse
theory  for closed geodesics is to establish the growth
of the Morse index by iteration of a given closed geodesic.
The deepest results in this direction for the Riemannian case are due to Bott
in the famous paper \cite{Bo4}; using complexifications and a suitable intersection theory,
Bott proves that all the information
on the Morse index and the nullity of the iterates of a given closed
Riemannian geodesic is encoded into two integer valued functions
on the unit circle.
Following Bott's ideas, in this paper we will prove
some iteration formulas for the Morse index of the critical points
of the constrained variational problem for stationary closed Lorentzian geodesics
mentioned above. More
precisely, given a closed geodesic $\gamma$, we will show (Theorem~\ref{thm:iterazioneindice})
the existence of an
integer valued function $\Lambda_\gamma$ on the circle $\mathds S^1$ with the
property that, given a closed geodesic $\gamma$, its  $N$-iterate $\gamma^N$ has Morse
index $\mu(\gamma^N)$ given by the sum of the values of $\Lambda_\gamma$
at the $N$-th roots of unity,
$k=1,\ldots,N$, plus a correction term $\epsilon_\gamma\in\{0,1\}$.
The difference $\mu_0(\gamma)=\mu(\gamma)-\epsilon_\gamma$, the \emph{restricted Morse index}
of $\gamma$, plays a central role
in the theory; it can be interpreted as the index of the second variation of the
geodesic action functional restricted to the space of variational vector fields arising
from variations of $\gamma$ by curves $\mu$ satisfying $g(\dot\mu,\mathcal K)=g(\dot\gamma,\mathcal K)\
\text{(constant)}$.

In analogy with the Riemannian case, $\Lambda_\gamma$ is a lower semi-continuous function on the circle
(except possibly at $1$ in a singular case mentioned below) and its jumps can occur only at points of
$\mathds S^1$ that belong to the spectrum of the (complexified) \emph{linearized Poincar\'e} map $\mathfrak
P_\gamma$ of $\gamma$. Given one such discontinuity point $\rho\in\mathds S^1$, the (complex)
dimension of  the kernel of $\mathfrak P_\gamma-\rho$ is an upper bound for the value of the
jump of $\Lambda_\gamma$ at $\rho$. Explicit, although extremely involved, computations can be
attempted in order to compute the precise value of each jump of $\Lambda_\gamma$
(see Subsection~\ref{sub:jumpsindexfnctn}). It may be interesting to observe here that
the question is reduced to an algebraic counting of the zeros in the spectrum (i.e., the \emph{spectral flow})
of an \emph{analytic} path of Fredholm self-adjoint operators, for which a finite dimensional reduction
 and a higher order method are available (see \cite{GPP}).

As a special case of our iteration formula, we show that if $\gamma$ is
\emph{strongly hyperbolic} (i.e., $\epsilon_\gamma=0$ and there is no eigenvalue of $\mathfrak P_\gamma$ on
the unit circle), then the restricted Morse index of $\gamma^N$ is equal to the
restricted Morse index of $\gamma$ multiplied by $N$. Also, the correction term
$\epsilon_{\gamma^N}$ coincides with $\epsilon_\gamma$ for all $N$.
As an application of this
fact, we will use an argument from equivariant Morse theory to prove (Proposition~\ref{thm:BTZ}) the existence of infinitely many
\emph{geometrically distinct} closed geodesics in a class of non
simply connected globally hyperbolic stationary spacetimes,
generalizing the results of \cite{BalThoZil}.

A second important application of the theory developed in this paper is the
proof of a \emph{uniform linear growth} for the index of an iterate (Proposition~\ref{eq:stimamigliore});
this is a crucial step in Gromoll and Meyer's result on the existence of infinitely
many closed geodesics in compact Riemannian manifolds.
The uniform estimate on the linear growth allows to prove that the
contribution to the homology of the free loop space in a fixed dimension
provided by the tower of iterates of a given closed geodesic
only depends on a uniformly bounded number of iterates.

Compared to the Riemannian case, several new phenomena appear in the stationary Lorentzian
case. In first place, the question of the correction term $\epsilon_\gamma$ is somewhat puzzling,
as this part of the index is not detected by the values of the function $\Lambda_\gamma$ on
$\mathds S^1\setminus\{1\}$. Its geometric interpretation is a little involved; roughly speaking,
$\epsilon_\gamma$ vanishes when $\gamma$ can be perturbed to a curve with less energy only by
curves that ``form a fixed angle'' with the timelike Killing field $\mathcal K$. Infinitesimally,
this amounts to saying that the index of the index form does not decrease when the form is restricted
to the space of variations $V$ satisfying $\mathfrak g(V',\mathcal K)-\mathfrak g(V,\smash{\mathcal K}')=0$,
where the prime denotes covariant differentiation along $\gamma$. Particularly significative is the fact
that the correction term $\epsilon_\gamma$ is the same for all the geodesics in the tower of iterates of $\gamma$.
An important class of examples of geodesics $\gamma$ with $\epsilon_\gamma=0$ is obtained by taking
geodesics that are everywhere orthogonal to $\mathcal K$ in the \emph{static} case,
i.e., when the orthogonal distribution $\smash{\mathcal K}^\perp$ to $\mathcal K$ is integrable
(see Example~\ref{exa:casostatico}).
In this case, every integral submanifold of $\smash{\mathcal K}^\perp$ is a Riemannian
totally geodesic hypersurface of $M$; this suggests that $\epsilon_\gamma$ can be interpreted
as a sort of measure of the ``non Riemannian behavior'' of the closed geodesic $\gamma$.
It is plausible to conjecture that the question of the distribution of conjugate points
along $\gamma$ be related to the value of $\epsilon_\gamma$; the results of an investigation in this
direction are left to a forthcoming paper.

A second peculiar phenomenon of the stationary Lorentzian case is the existence of
a singular class of closed geodesics $\gamma$, that are characterized by the fact that
the covariant derivative $\smash{\mathcal K}'$ along $\gamma$ is pointwise multiple
of the projection of $\mathcal K$ onto the orthogonal space $\dot\gamma^\perp$;
this includes in particular all closed geodesics along which
the Killing field is parallel. As it is shown in \cite{JMP07}, the fact that $\mathcal K$ is singular along $\gamma$  is equivalent to the existence of a family of parallel vectorfields along the geodesics that generate $\mathcal K(s)^\bot$ for every $s\in [0,1]$. This is true in particular when the geodesic is contained in a totally geodesic hypersurface orthogonal to $\mathcal K$. Again, orthogonal geodesics are singular in the static
case. When $\gamma$ is singular, the question of semi-continuity of the function $\Lambda_\gamma$
is more delicate, and, in fact, it may fail to hold at the point $1$ if $\epsilon_\gamma=1$, even
when the linearized Poincar\'e map of $\gamma$ does not contain $1$ in its spectrum.

As already observed in \cite{Bo4}, in spite of the initial geometrical motivation the
theory developed in the present paper is better cast in the language of Morse--Sturm differential
systems in the complex space $\mathds C^n$. We will first discuss our results in this
abstract setup (Section~\ref{sec:morsesturm}); the reduction of the stationary Lorentzian
geodesic problem  to a Morse--Sturm system is done via a parallel (not necessarily periodic)
trivialization of the orthogonal bundle $\dot\gamma^\perp$ along the geodesic $\gamma$
(Subsection~\ref{sub:geomorsesturm}). As to this point, it is interesting to observe
here that, if on one hand to use a parallel trivialization simplifies the corresponding
Morse--Sturm system and its index form (see \eqref{eq:defMorseSturm}, \eqref{eq:defIndexformCn}),
on the other hand the lack of periodicity of such trivializations imposes the introduction
of more involved boundary conditions. In our notations, information on the boundary
conditions is encoded in the endomorphism $T$ (see Subsection~\ref{sub:basicsetup}),
which represents the parallel transport
along the geodesic. When the geodesic is orientable, then $T$ is (the complex extension
of an isomoprhism) in the connected component of the identity of $\mathrm{GL}(\R^{2n})$.
Using this setup, no distinction is necessary between orientable and non orientable closed
geodesics; recall that these two cases are distinguished in Bott's original work,
the non orientable ones corresponding to the $N$-th roots of $-1$. We observe that the theory developed in this work can be set in a more general background than stationary spacetimes, that is, when considering geodesics in general Lorentzian manifolds admiting a periodic timelike Jacobi field.
\smallskip

The paper is organized as follows. The main technical results, presented in the context
of abstract Morse--Sturm systems in $\C^n$, are discussed in the first part of the
paper (Sections~\ref{sec:morsesturm}, \ref{sec:nullity} and \ref{sec:sequence}), while the applications to closed
geodesics are discussed in Section~\ref{sec:closedgeodesics}.
In Section~\ref{sec:morsesturm} we set up the basis of the theory, with a description of
a class of complex Morse--Sturm systems that are symmetric with respect to a nondegenerate
Hermitian form of index $1$ in $\C^n$, their index form, and with the description of
two families of closed subspaces of the Sobolev space $H^1\big([0,1],\C^n\big)$.
These spaces are parameterized by unit complex numbers, and a central point
is determining their continuity with respect to the parameter. Although only the continuity of
these spaces is actually required in our theory, we will  show
that the dependence  is in fact analytic. The purpose of this fact
is that, in view to future developments, one might attempt to use higher order methods
for determining the value of the jumps of the index function (see Subsection~\ref{sub:jumpsindexfnctn}),
which require analyticity of the eigenvalues and of the eigenvectors of the corresponding
self-adjoint operators (see \cite{GPP}).
At this stage, this seems to be a rather involved question, that will be treated
only marginally in this paper.

In Section~\ref{sec:nullity} we use a functional analytical approach to determine the
kernel of the index form restricted to the family of closed subspaces mentioned above.
Section~\ref{sec:sequence} contains the main technical result of the paper (Proposition~\ref{thm:prop4.2}),
which is a formula relating the index of the $N$-iterate with the sum of the indexes of the index form
at the space of vector fields satisfying boundary conditions involving the $N$-th roots of
unity. Following Bott's suggestive terminology, we call this the \emph{Fourier theorem}.
In Section~\ref{sec:closedgeodesics} the theory is applied to the case of closed
geodesics in stationary Lorentzian manifolds, with some emphasis to the static case
which provides interesting examples of singular solutions.
Propositions~\ref{eq:stimamigliore} and \ref{thm:BTZ} are two direct applications of
our results to the global theory of closed Lorentzian geodesics, which is the final
objective of our research.
Finally, we conclude with a short section containing remarks, conjectures and
suggestions for future developments.

\end{section}

\begin{section}{On a class of non positive definite Morse--Sturm systems in $\mathds C^n$}
\label{sec:morsesturm}
\subsection{The basic setup}\label{sub:basicsetup}
Let us consider given the following objects:
\begin{itemize}
\item[(a)] $n$ is an integer greater than or equal to $1$
\item[(b)] $g$ is a nondegenerate symmetric bilinear form on $\R^n$  having index $1$, extended by sesquilinearity
to a nondegenerate Hermitian form on $\C^n$;
\item[(c)] $T:\mathds R^n\to\mathds R^n$ is a $g$-preserving linear isomorphism of $\mathds R^n$, extended by complex linearity
to a $g$-preserving isomorphism of $\C^n$;
\item[(d)] $[0,1]\ni t\mapsto R(t)\in\mathrm{End}(\mathds R^n)$ is a continuous map of $g$-symmetric
(i.e., $gR(t)=R(t)^*g$) linear endomorphisms of $\mathds R^n$ satisfying:
\[R(0)T=TR(1);\]
$R(t)$ is extended by complex linearity to a $g$-Hermitian endomorphism of $\C^n$.
\item[(e)] $Y:[0,1]\to\mathds R^n\subset\C^n$ is a $C^2$-solution of the \emph{Morse--Sturm system}:
\begin{equation}\label{eq:defMorseSturm}
V''(t)=R(t)V(t)
\end{equation}
that satisfies:
\begin{itemize}
\item[(e1)] $g(Y,Y)<0$ everywhere on $[0,1]$;
\item[(e2)] $TY(1)=Y(0)$ and $TY'(1)=Y'(0)$.
\end{itemize}
\end{itemize}
The solution $Y$ of \eqref{eq:defMorseSturm} will be called \emph{singular} if $Y'(s)$ is a multiple of $Y(s)$ for all
$s\in[0,1]$; the singularity of $Y$ is equivalent to either one of the following two conditions:
\begin{equation}\label{eq:defYsingular}
g(Y,Y)\,Y'=g(Y',Y)\,Y,\quad\text{or}\quad \left[\frac Y{g(Y,Y)}\right]'+\frac {Y'}{g(Y,Y)}=0\quad\text{on $[0,1]$.}
\end{equation}
The $g$-symmetry of $R$ implies that, given any two solutions $V_1$ and $V_2$ of \eqref{eq:defMorseSturm}, then
the quantity $g(V_1',V_2)-g(V_1,V_2')$ is constant on $[0,1]$:
\begin{equation}\label{eq:conservazione}
\dds\big[g(V_1',V_2)-g(V_1,V_2')\big]=g(V_1'',V_2)-g(V_1,V_2'')=g(RV_1,V_2)-g(V_1,RV_2)=0.
\end{equation}
We will consider extensions to the real line $Y:\R\to\C^n$ and $R:\R\to\mathrm{End}(\C^n)$
of the maps $Y$ and $R$ above by setting:
\begin{equation}\label{eq:2.1b}
Y(t+N)=T^{-N}Y(t),\quad R(t+N)=T^{-N}R(t)T^N,\quad\forall\,t\in\left[0,1\right[,\ N\in\Z;
\end{equation}
having this in mind, we also set $R_N(t)=R(tN)$ and $Y_N(t)=Y(tN)$ for all $t\in[0,1]$.
Observe that $Y_N$ is of class $C^2$ and $R_N$ is continuous on $[0,1]$, moreover, from \eqref{eq:2.1b}
one gets easily:
\begin{equation}\label{eq:periodYNeRN}
Y_N(t+k/N)=T^{-k}Y_N(t),\qquad R_N(t+k/N)=T^{-k}R_N(t)T^k,
\end{equation}
for every $k\in\Z$.
The \emph{$N$-th iterated} of the Morse--Sturm system \eqref{eq:defMorseSturm} is the Morse--Sturm
system:
\begin{equation}\label{eq:iteratedMS}
V''(t)=N^2R_N(t)V(t).
\end{equation}
If $Y$ is a singular solution of \eqref{eq:defMorseSturm}, then $Y_N$ is a singular solution of \eqref{eq:iteratedMS},
in which case equalities \eqref{eq:defYsingular} hold with $Y$ replaced by $Y_N$.
We will consider the following additional data.
\subsection{The index forms}
Let $\mathcal H$ be the Hilbert space
$H^1\big([0,1],\C^n\big)$ of $\C^n$-valued maps on the interval $[0,1]$ and of Sobolev class
$H^1$; moreover, for all $N\ge1$  let $I_N:\mathcal H\times\mathcal H\to\C$ be the bounded sesquilinear form:
\begin{equation}\label{eq:defIndexformCn}
I_N(V,W)=\int_0^1\big[g(V',W')+N^2g(R_NV,W)\big]\,\mathrm dt,\qquad V,W\in\mathcal H.
\end{equation}

We will also introduce a smooth family of positive definite Hermitian forms $g_t^{N}$ on $\C^n$, defined
using $Y$ by:
\begin{equation}
g_t^{N}(V,W)=g(V,W)-2\,\frac{g\big(V,Y_N(t)\big)\cdot g\big(W,Y_N(t)\big)}{g\big(Y_N(t),Y_N(t)\big)};
\end{equation}
denote by $A:[0,1]\times\N\to \mathcal{L}(\C^n)$ the smooth family of symmetric isomorphisms such that
\begin{equation}\label{eq:defAtN}
g(V,W)=g_t^{N}\big(A(t,N)V,W\big)
\end{equation}
for every $V,W\in\C^n.$
We will think of $\mathcal H$ endowed with a family of  Hilbert space inner products:
\begin{equation}\label{eq:defIP}
\llangle V,W\rrangle_N=\int_0^1\big[ g_t^{N}(V',W')+g_t^{N}(V,W)\big]\,\mathrm dt.
\end{equation}

\subsection{Analytic families of closed subspaces}\label{sub:closedsubspaces}
We will now define a family of closed subspaces of $\mathcal H$, as follows.
Let $\mathds S^1$ denote the set of unit complex numbers; for $\rho\in\mathds S^1$ and $N\ge1$, set:
\[\mathcal H^\rho(N)=\big\{ V\in\mathcal H: T^NV(1)=\rho^NV(0)\big\},\]
\[\mathcal H_*(N)=\Big\{V\in\mathcal H:g(V',Y_N)-g(V,Y_N')=C_V\ \text{(constant)\ a.e.\ on\ [0,1]}\Big\},\]
and:
\[\mathcal H_0(N)=\Big\{V\in\mathcal H:g(V',Y_N)-g(V,Y_N')=0\ \text{a.e.\ on}\ [0,1]\Big\}.\]
Finally, define:
\[\mathcal H_*^\rho(N)=\mathcal H_*(N)\cap\mathcal H^\rho(N),\quad \mathcal H_0^\rho(N)=\mathcal H_0(N)\cap\mathcal H^\rho(N).\]

\begin{prop}\label{thm:ker*}
The kernel of the restriction of $I_N$ to $\mathcal H^\rho(N)$ coincides with the kernel of the restriction
of $I_N$ to $\mathcal H_*^\rho(N)$, and it is given by the finite dimensional space:
\begin{equation}\label{eq:nucleo}
\Big\{\!V\in C^2\big([0,1],\C^n\big):  \text{$V$ solution of \eqref{eq:iteratedMS}},\ T^NV(1)=\rho^NV(0),\ T^NV'(1)=\rho^NV'(0)\!\Big\}.
\end{equation}
\end{prop}
\begin{proof}
That \eqref{eq:nucleo} is the kernel of $I_N$ in $\mathcal H^\rho(N)$ follows easily from a partial
integration in \eqref{eq:defIndexformCn}. Denote by $\mathfrak Y(N)$ the subspace of $\mathcal H^\rho(N)$
consisting of vector fields of the form $f\cdot Y_N$, where $f\in H^1\big([0,1],\C^n\big)$ is such that
$f(0)=f(1)=0$. The conclusion follows easily from the fact that $\mathcal H^\rho(N)=\mathcal H_*^\rho(N)+\mathfrak Y^\rho(N)$,
that the spaces $\mathcal H_*^\rho(N)$ and $\mathfrak Y^\rho(N)$ are $I_N$-orthogonal, that $I_N$ is negative definite on
$\mathfrak Y(N)$, and that
\eqref{eq:nucleo} is contained in $\mathcal H_*^\rho(N)$.
\end{proof}

\begin{cor}\label{thm:corker0ker*}
If $\rho$ is not an $N$-th root of unity, then $\Ker\big(I_N\vert_{\mathcal H_*^\rho(N)\times\mathcal H_*^\rho(N)}\big)\subset\mathcal H_0^\rho(N)$.
\end{cor}
\begin{proof}
Recalling that $T$ is $g$-symmetric, we have:
\begin{multline*}
g\big(V'(0),Y_N(0)\big)-g\big(V(0),Y_N'(0)\big)=g\big(V'(1),Y_N(1)\big)-g\big(V(1),Y_N'(1)\big)\\=\rho^N\big[g\big(V'(0),Y_N(0)\big)-g\big(V(0),Y_N'(0)\big)\big],
\end{multline*}
from which the conclusion follows.
\end{proof}
Let us introduce the following:
\begin{defin}\label{thm:defC1subspaces}
Let $\mathfrak H$ be a complex Hilbert space, $I\subset\R$ an interval and
$\{\mathcal D^t\}_{t\in I}$ be a family of closed subspaces of
$\mathfrak H$. We say that $\{\mathcal D^t\}_{t\in I}$ is a \emph{$C^k$-family}, $k=0,\ldots,\infty$, (resp., an {\em
analytic family\/}) of subspaces if for all $t_0\in I$ there exist
$\varepsilon>0$, a $C^k$ (resp., an analytic) curve
$\alpha:\left]t_0-\varepsilon,t_0+\varepsilon\right[\cap I\mapsto
\mathcal L(\mathfrak H)$ and a closed subspace $\overline{\mathcal
D}\subset\mathfrak H$ such that $\alpha(t)$ is an isomorphism  and
$\alpha(t)(\mathcal D^t)=\overline{\mathcal D}$ for all~$t$.
\end{defin}
Definition~\ref{thm:defC1subspaces} is generalized obviously to the case
of families $\{\mathcal D^\theta\}_{\theta\in\mathds S^1}$ parameterized
on the circle.
Let us give a criterion for the smoothness
of a family of closed subspaces:
\begin{prop}\label{thm:produce}
Let $I\subset\R$ be an interval, $\mathfrak H,\tilde{\mathfrak H}$ be Hilbert
spaces and $F:I\mapsto\mathcal L(\mathfrak H,\tilde{\mathfrak H})$ be a $C^k$ (resp., analytic) map
such that each $F(t)$ is surjective. Then, the family $\mathcal D^t=\mathrm{Ker}\big(F(t)\big)$
is a $C^k$-family (resp., an analytic family) of closed subspaces of $\mathfrak H$.
\end{prop}
\begin{proof}
See for instance \cite[Lemma~2.9]{asian}.
\end{proof}
Clearly, if $\{\mathcal D^t\}_{t\in I}$ is a $C^k$ (resp., an analytic) family of closed subspaces of $\mathfrak H_1$ and
$\mathfrak H_1$ is a closed subspace of $\mathfrak H$, then $\{\mathcal D^t\}_{t\in I}$ is a $C^k$ (resp., an analytic) family
of closed subspaces of $\mathfrak H$. It is also clear that, given a $C^k$ (resp., an analytic) family of closed
subspaces $\{\mathcal D^t\}_{t\in I}$ of $\mathfrak H$ and given a $C^k$ (resp., an analytic) map $t\mapsto\psi_t$ of
isomorphisms of $\mathfrak H$, then $\{\psi_t(\mathcal D^t)\}_{t\in I}$ isa $C^k$ (resp., an analytic) family
of closed subspaces of $\mathfrak H$. We will need later a slight improvement of
Proposition~\ref{thm:produce}:
\begin{cor}\label{thm:imprproduce}
If $\{\mathcal D^t\}_{t\in I}$ is a $C^k$ (resp., an analytic) family of closed subspaces of
$\mathfrak H$, and if $F:I\mapsto\mathcal L(\mathfrak H,\tilde{\mathfrak H})$
is a $C^k$ (resp., an analytic) map such that the restriction of $F(t)$ to $\mathcal D^t$ is
surjective for all $t\in I$, then $\mathcal E_t=\Ker\big(F(t)\big)\cap\mathcal D^t$ is
a $C^k$ (resp., an analytic) family of closed subspaces of $\mathfrak H$.
\end{cor}
\begin{proof}
Let $\alpha:\left]t_0-\varepsilon,t_0+\varepsilon\right[\cap I\mapsto
\mathcal L(\mathfrak H)$ be a local trivialization of $\{\mathcal D^t\}$,
$\alpha(t)(\mathcal D^t)=\overline{\mathcal D}$ and
consider the $C^k$ (analytic) map $t\mapsto\widetilde F(t):\overline{\mathcal D}\to\widetilde{\mathfrak H}$
given by $\widetilde F(t)=F(t)\circ\left(\alpha(t)\vert_{\mathcal D^t}\right)^{-1}$.
Since the restriction of $F(t)$ to $\mathcal D^t$ is surjective, then $\widetilde F(t)$ is
surjective. By Proposition~\ref{thm:produce}, the family $\Ker\big(\widetilde F(t)\big)=\alpha(t)\big(\mathcal E_t\big)$
is a $C^k$ (resp., an analytic) family of closed subspaces of $\overline{\mathcal D}$. It follows that
$\mathcal E_t=\alpha(t)^{-1}\big[\Ker\big(\widetilde F(t)\big)\big]$ is a $C^k$ (resp., an analytic) family of closed
subspaces of $\mathfrak H$.
\end{proof}
\begin{prop}\label{thm:corcriteriosmoothness}
For all $N\ge1$,  the collection $\big\{\mathcal H_*^\rho(N)\big\}_{\rho\in\mathds S^1}$
is an  analytic family of closed subspaces of $\mathcal H$. If $Y$ is not singular, then
the same conclusion holds also for the family $\big\{\mathcal H_0^\rho(N)\big\}_{\rho\in\mathds S^1}$.
\end{prop}
\begin{proof}
Consider the analytic map $\mathds S^1\ni\rho\mapsto F_\rho$, $F_\rho:\mathcal H\to\C^n$ given by
$F_\rho(V)=T^NV(1)-\rho^NV(0)$; in order to apply Proposition~\ref{thm:produce} we need to show
that the restriction of $F_\rho$ to $\mathcal H_*^\rho(N)$
is surjective for all $\rho\in\mathds S^1$. Clearly, $F_\rho:\mathcal H\to\C^n$ is surjective.
Given $W\in\mathcal H$, there exists $V\in\mathcal H_*^\rho(N)$ with $V(0)=W(0)$ and
$V(1)=W(1)$. Such $V$ is obtained by setting $V(s)=W(s)+f_W(s)\cdot Y_N(s)$, $s\in[0,1]$,
where \[f_W(s)=\int_0^s\frac{C+g(W,Y_N')-g(W',Y_N)}{g(Y_N,Y_N)}\,\mathrm dr\]
and
\begin{equation}\label{eq:defCWcompl}
C=\left(\int_0^1\frac{\mathrm dr}{g(Y_N,Y_N)}\right)^{-1}\int_0^1\frac{g(W',Y_N)-g(W,Y_N')}{g(Y_N,Y_N)}\,\mathrm dr.
\end{equation}
It is easily seen that such $V$ satisfies $g(V',Y_N)-g(V,Y_N')\equiv C$, the obvious details of such computation being omitted.
Since in this situation $F_\rho(V)=F_\rho(W)$, it follows immediately that also the restriction of $F_\rho$ to
$\mathcal H_*^\rho(N)$ is surjective.

Consider now the case of the family $\mathcal H_0^\rho(N)$; the non singularity assumption on $Y$
implies that we can find a subinterval $[a,b]\subset\left]0,1\right[$ such that
\begin{equation}\label{eq:equivY'multY}
\mathbf T_N:=\left[\frac{Y_N}{g(Y_N,Y_N)}\right]'+\frac{Y_N'}{g(Y_N,Y_N)}\ne0\quad\text{on $[a,b]$}.
\end{equation}
In this case, we show that the restriction of $F_\rho$ to $\mathcal H_0^\rho(N)$ is surjective by showing
that for all $Z_0,Z_1\in\C^n$  there exists $W\in\mathcal H$ with $W(0)=Z_0$, $W(1)=Z_1$ and such that
the quantity $C$ given in \eqref{eq:defCWcompl} vanishes.
For, choose arbitrary smooth  maps
$t_1:[0,a]\to\C^n$ and $t_2:[b,1]\to\C^n$ such that $t_1(0)=Z_0$, $t_2(1)=Z_1$ and $t_1(a)=t_2(b)=0$.
The desired $W$ is then obtained by setting:
\[W(s)=\begin{cases}t_1(s)&\text{if $s\in\left[0,a\right[$,}\\
h(s)&\text{if $s\in[a,b]$,}\\
t_2(s)&\text{if $s\in\left]b,1\right]$,}\end{cases}\]
where $h\in H^1_0\big([a,b];\C^n\big)$ is to be chosen in such a way that
\begin{multline}\label{eq:equazh}\int_a^b
\frac{g(h',Y_N)-g(h,Y_N')}{g(Y_N,Y_N)}\,\mathrm dr=\\
-\int_0^a\frac{g(t_1',Y_N)-g(t_1,Y_N')}{g(Y_N,Y_N)}\,\mathrm dr-\int_b^1\frac{g(t_2',Y_N)-g(t_2,Y_N')}{g(Y_N,Y_N)}\,\mathrm dr.
\end{multline}
The left hand side of this equality defines a bounded linear functional on $H^1_0\big([a,b],\C^n\big)$ which is not null;
this is easily seen using partial integration:
\[\int_a^b
\frac{g(h',Y_N)-g(h,Y_N')}{g(Y_N,Y_N)}\,\mathrm dr=-\int_a^bg\big(h,\mathbf T_N\big)\,\mathrm dr,\]
and using our assumption \eqref{eq:equivY'multY}. In particular, a function $h\in H^1_0\big([a,b],\C^n\big)$
satisfying \eqref{eq:equazh} can be found, which concludes the argument.
\end{proof}

\subsection{Singular solutions}
Let us now assume that $Y$ is a singular solution of the Morse--Sturm system \eqref{eq:defMorseSturm},
which is equivalent to assuming that the maps $\mathbf T_N$ defined in \eqref{eq:equivY'multY} vanish
identically on $[0,1]$ for all $N\ge1$.
\begin{lem}\label{thm:spazicoincsing}
If $Y$ is a singular solution of \eqref{eq:defMorseSturm} and $\rho\in\mathds S^1$ is  an $N$-th root
of unity, then  $\mathcal H_0^\rho(N)=\mathcal H_*^\rho(N)$.
\end{lem}
\begin{proof} If $V\in \mathcal H_*^\rho(N)$, a direct computation gives:
\begin{multline}\label{eq:acosaeugualeCv}
C_V\int_0^1\frac{\mathrm dr}{g(Y_N,Y_N)}=\int_0^1\frac{g(V',Y_N)-g(V,Y_N')}{g(Y_N,Y_N)}\,\mathrm dr\\
=\frac{g(V,Y_N)}{g(Y_N,Y_N)}\Big\vert_0^1-\int_0^1g(V,\mathbf T_N)\,\mathrm dr=\frac{g(V,Y_N)}{g(Y_N,Y_N)}\Big\vert_0^1=(\rho^N-1)\frac{g\big(V(0),Y_N(0)\big)}{g\big(Y_N(0),Y_N(0)\big)},
\end{multline}
from which the conclusion follows easily.
\end{proof}
Let us show now that the $N$-th roots of unity are the unique discontinuities of the family $\big\{\mathcal H_0^\rho(N)\big\}$.
\begin{prop}\label{thm:hoanalytic}
If $Y$ is a singular solution of \eqref{eq:defMorseSturm} and $\mathcal A\subset\mathds S^1$ is a connected
subset that does not contain any $N$-th root of unity, the family $\big\{\mathcal H_0^\rho(N)\big\}_{\rho\in\mathcal A}$
is an analytic family of closed subspaces of $\mathcal H$.
\end{prop}
\begin{proof}
We use Corollary~\ref{thm:imprproduce} applied to the analytic family $\mathcal H^\rho_*(N)$
and the constant map $F(t)=F_N:\mathcal H\to\C$ defined by:
\[F_N(V)=\left(\int_0^1\frac{\mathrm dr}{g(Y_N,Y_N)}\right)^{-1}\int_0^1\frac{g(V',Y_N)-g(V,Y_N')}{g(Y_N,Y_N)}\,\mathrm dr.\]
The restriction of $F_N$ to $\mathcal H^\rho_*(N)$ is the map $V\mapsto C_V$;
by \eqref{eq:acosaeugualeCv}, such restriction is surjective (i.e., not identically zero)
when $\rho$ is not an $N$-th root of unity. Observe indeed that, as it follows easily arguing as in the proof
of Proposition~\ref{thm:corcriteriosmoothness}, $V(0)$ is an arbitrary vector of $\C^n$ when $V$ varies in
$\mathcal H^\rho_*(N)$.
\end{proof}
In particular, we have the following:
\begin{cor}\label{thm:cor2.9}
If $Y$ is a singular solution of \eqref{eq:defMorseSturm}, then the collection $\{\mathcal H_0^\rho(1)\}_{\rho\in\mathds S^1\setminus\{1\}}$
is an analytic family of closed subspaces of $\mathcal H$.\qed
\end{cor}
\subsection{Finiteness of the index}
\begin{prop}\label{thm:prop2.10}
For all $N\ge1$ and for all $\rho\in\mathds S^1$, the restriction of $I_N$ to $\mathcal H_*^\rho(N)\times\mathcal H_*^\rho(N)$
is \emph{essentially positive}, i.e., it is represented (relatively to the inner product \eqref{eq:defIP})
by a self-adjoint operator on $\mathcal H_*^\rho(N)$ which is a compact perturbation of a positive isomorphism.
In particular, the index of $I_N$ on $\mathcal H_*^\rho(N)$ is finite.
\end{prop}
\begin{proof}
We will show that the restriction of $I_N$ to $\mathcal H_*^\rho(N)$ is the sum of the inner product
$\llangle\cdot,\cdot\rrangle_N$ and a symmetric bilinear form $B$ which is continuous relatively
to the $C^0$-topology. The conclusion will follow from the fact that the inclusion of $H^1$ into
$C^0$ is compact, and therefore $B$ is represented by a compact operator.

The linear map $\mathcal H^\rho_*(N)\ni V\mapsto C_V\in\mathds C$ is continuous relatively to the $C^0$-topology, for:
\[C_V=(\rho^N-1)g\big(V(0),Y_N(0)\big)-2\int_0^1g(V,Y_N')\,\mathrm dt.\]
A straightforward calculation shows that, for $V,W\in\mathcal H_*^\rho(N)$, $I_N(V,W)$ can be
written as:
\[I_N(V,W)=\llangle V,W\rrangle_N+\int_0^1\Big[2\,\frac{[g(V,Y_N')+C_V]\cdot[g(W,Y_N')+C_W]}{g(Y_N,Y_N)}-g_t^N(V,W)\Big]\,\mathrm dt,\]
from which the conclusion follows easily.
\end{proof}
\begin{cor}
\label{thm:contindex}
Let $\mathcal A\subset\mathds S^1$ be a connected subset such that the restriction of $I_N$ to $\mathcal H^\rho_*(N)$
is nondegenerate for all $\rho\in\mathcal A$. Then, the index of such restriction
is constant on $\mathcal A$. Similarly, if $Y$ is not a singular solution of \eqref{eq:defMorseSturm},
the same result holds for the restriction of $I_N$ to $\mathcal H^\rho_0(N)$; if $Y$ is
a singular solution of \eqref{eq:defMorseSturm}, then the result holds under the additional assumption
that $\mathcal A$ does not contain any $N$-th root of unity.
\end{cor}
\begin{proof}
By continuity, the jumps of the map $\mathds S^1\ni\rho\mapsto\mathrm n_-\big(I_N\vert_{\mathcal H^\rho_*(N)\times \mathcal H^\rho_*(N)}\big)\in\N$
can only occur at those points $\rho$ where $I_N$ is degenerate on $\mathcal H^\rho_*(N)$.
The case of $\mathcal H^\rho_0(N)$ is analogous, using Proposition~\ref{thm:hoanalytic}.
\end{proof}
The discontinuities of the index function will be studied in Subsection~\ref{sub:jumpsindexfnctn} below.
\subsection{The linear Poincar\'e map}\label{sub:linearpoicare}
The last ingredient of our theory is the linear map $\mathfrak P:\C^n\oplus\C^n\to\C^n\oplus\C^n$ defined by:
\[\mathfrak P(v,w)=\big(TJ(1),TJ'(1)\big),\]
where $J:[0,1]\to\C^n$ is the (unique) solution of the Morse--Sturm system \eqref{eq:defMorseSturm} satisfying the initial
conditions $J(0)=v$ and $J'(0)=w$.
Using \eqref{eq:2.1b} one sees immediately that, given $N\ge1$, the $N$-th power $\mathfrak P^N$ is given
by:
\[\mathfrak P^N(v,w)=\big(T^NJ(1),T^NJ'(1)\big),\]
where $J$ is the solution of the equation $J''=N^2R_NJ$ satisfying $J(0)=v$ and $J'(0)=w$.
We will call $\mathfrak P$ the \emph{linear Poincar\'e map} of the Morse--Sturm system \eqref{eq:defMorseSturm};
clearly, $\mathfrak P$ is the complex linear extension of an endomorphism of $\R^n\oplus\R^n$ defined
using the real Morse--Sturm system. In particular, the spectrum $\mathfrak s(\mathfrak P)$ of $\mathfrak P$ is closed by conjugation.
\begin{prop}\label{thm:kerneleigenspace}
For all $\rho\in\mathds S^1$ and all $N\ge1$, the map $V\mapsto \big(V(0),V'(0)\big)$ gives an isomorphism from the kernel of the restriction of $I_N$ to
$\mathcal H_*^\rho(N)$ onto the $\rho^N$-eigenspace of $\mathfrak P^N$.
\end{prop}
\begin{proof}It
follows immediately from Proposition~\ref{thm:ker*}.
\end{proof}

The \emph{restricted} linear Poincar\'e map $\mathfrak P_0$ is the complex linear extension
of the restriction of $\mathfrak P$ to the invariant subspace $\mathfrak J_0\subset\R^n\oplus\R^n$ defined
by:
\[\mathfrak J_0=\big\{(v,w)\in\R^n\oplus\R^n:g\big(w,Y(0)\big)-g\big(v,Y'(0)\big)=0\big\}.\]
The invariance of $\mathfrak J_0$ is easily established using \eqref{eq:conservazione} and the equalities $Y(0)=TY(1)$, $Y'(0)=TY'(1)$.
Clearly, $\mathfrak s(\mathfrak P_0)\subset\mathfrak s(\mathfrak P)$; actually, the following holds:
\begin{lem}
$\mathfrak s(\mathfrak P_0)\setminus\{1\}=\mathfrak s(\mathfrak P)\setminus\{1\}$.
\end{lem}
\begin{proof}
As in Corollary~\ref{thm:corker0ker*}.
\end{proof}

\subsection{The index sequences and the nullity sequences}\label{sub:nulindseq}
Recall that, given a symmetric bilinear form $B:V\times V\to\R$ on a real vector space, the index and the nullity
of $B$ are defined respectively as the dimension of a maximal subspace on which $B$ is negative definite, and
the dimension of the kernel of $B$. Similarly, one defines index and nullity of a Hermitian sesquilinear bilinear
form on a complex vector space; the index and the nullity of a symmetric bilinear form on a real vector space
$V$ are equal respectively to the index and the nullity of the sesquilinear extension of $B$ to the complexification
of $V$.
\begin{defin}
For all $\rho\in\mathds S^1$, define the sequences:
\begin{align*}
&\lambda_*(\rho,N)=\text{index of $I_N$ on $\mathcal H_*^\rho(N)$},\quad
\lambda_0(\rho,N)=\text{index of $I_N$ on $\mathcal H_0^\rho(N)$},\\
&\nu_*(\rho,N)=\text{nullity of $I_N$ on $\mathcal H_*^\rho(N)$},\quad
\nu_0(\rho,N)=\text{nullity of $I_N$ on $\mathcal H_0^\rho(N)$},
\end{align*}
where $N\ge1$.
\end{defin}
Clearly:
\[\lambda_0(\rho,N)\le\lambda_*(\rho,N)\le\lambda_0(\rho,N)+1\]
for all $N\ge1$ and all $\rho\in\mathds S^1$. By Corollary~\ref{thm:corker0ker*},
we have $\nu_*(\rho,N)\le\nu_0(\rho,N)$ when $\rho$ is not an $N$-th root of unity;
we will show later (Corollary~\ref{thm:corstimanu}) that
$\nu_0(\rho,1)\le\nu_*(\rho,1)$.
Corollary~\ref{thm:contindex} above says that the maps $\rho\mapsto\lambda_*(\rho,N)$ are constant
on connected subsets of the circle where $\nu_*(\rho,N)$ vanishes.

The theory developed so far gives us the following properties of the index and the nullity sequences:
\begin{prop}\label{thm:prop2.15}
For all $N\ge1$, the following statements hold.
\begin{itemize}
\item[(a)] The map $\rho\mapsto\lambda_*(\rho,N)$ is lower semi-continuous on $\mathds S^1$,
and so is $\rho\mapsto\lambda_0(\rho,N)$ if $Y$ is not a singular solution of \eqref{eq:defMorseSturm};
if $Y$ is a singular solution, then the map $\rho\mapsto\lambda_0(\rho,N)$ is lower semi-continuous on every connected
component of $\mathds S^1$ that does not contain $N$-th roots of unity.
\item[(b)]  $\nu_*(\rho,N)=\Dim\big(\Ker(\mathfrak P^N-\rho^N)\big)$.
\item[(c)] $\lambda_*(\rho,N)=\lambda_*(\bar\rho,N)$,  $\nu_*(\rho,N)=\nu_*(\bar\rho,N)$, $\lambda_0(\rho,N)=\lambda_*(\bar\rho,N)$,
and\hfill\break $\nu_0(\rho,N)=\nu_*(\bar\rho,N)$.
\end{itemize}
\end{prop}
\begin{proof}
The lower continuity of $\lambda_*$ and $\lambda_0$ claimed in part (a) follows easily from
the continuity of the family of subspaces $\mathcal H_*^\rho(N)$ and
$\mathcal H_0^\rho(N)$, which was  proved in Proposition~\ref{thm:corcriteriosmoothness}
for  the non singular case, and in Proposition~\ref{thm:hoanalytic} for the singular case.
Part (b) is a restatement of Proposition~\ref{thm:kerneleigenspace}.
For part (c), it suffices to observe that the map $V\mapsto\overline V$ (pointwise complex conjugation)
sends isomorphically the kernel (resp., a maximal negative subspace) of $I_N\vert_{\mathcal H_*^\rho(N)\times\mathcal H_*^\rho(N)}$
onto the kernel (resp., a maximal negative subspace) of $I_N\vert_{\mathcal H_*^{\overline\rho}(N)\times\mathcal H_*^{\overline\rho}(N)}$.
Similarly for $I_N\vert_{\mathcal H_0^\rho(N)\times\mathcal H_0^\rho(N)}$.
\end{proof}
\end{section}

\begin{section}{On the nullity sequences}\label{sec:nullity}
The aim of this section is to study the kernel of the restriction of the bilinear form $I=I_1$
to the space $\mathcal H^\rho_0$. We will perform this task by determining a differential equation
satisfied by vector fields $V_\rho$ that are eigenvectors of the restriction of $I_1$ to $\mathcal H^\rho_0$;
this is obtained by functional analytical techniques. The kernel of such restriction is obtained
as a special case when the eigenvalue is zero.
It is convenient to treat this subject using an $L^2$-approach (this facilitates the computation of adjoint maps),
and for this one must enter in the realm of unbounded operators. The following notation will be used:
\[\mathfrak K= L^2\big([0,1],\C^n\big),\]
\[\mathfrak K_*=\Big\{V\in\mathfrak K:g(V,Y)-2\int_0^tg(V,Y')\,\mathrm ds=2t\int_0^1 g(V,Y')\,\de s \ \text{a.e.\ on}\ [0,1]\Big\},\]
and:
\[\mathfrak K_0=\Big\{V\in\mathfrak K:g(V,Y)-2\int_0^tg(V,Y')\,\mathrm ds=0\ \text{a.e.\ on}\ [0,1]\Big\}.\]
We want to describe the orthogonal subspaces to $\mathfrak K_*$ and $\mathfrak K_0$ in $\mathfrak K$ relatively to the inner product
\begin{equation}\label{eq:definnerprodK}
\langle V,W\rangle_{\mathfrak K}=\int_0^1 g_t^{(r)}( V,W)\,\de t.
\end{equation}
The subspaces $\mathfrak K_*$ and $\mathfrak K_0$ are the kernels respectively of the bounded linear operators $T_*:\mathfrak K\to L^2([0,1];\C)$ given by
\begin{equation}
T_*(V)(t)=g(V(t),Y(t))-2\int^t_0 g(V,Y')\de s-2t\int_0^1 g(V,Y')\,\de s,
\end{equation}
and $T_0:\mathfrak K\to L^2\big([0,1];\C\big)$ given by
\begin{equation}
T_0(V)(t)=g\big(V(t),Y(t)\big)-2\int^t_0 g(V,Y')\,\de s.
\end{equation}
\begin{lem}
The operators $T_*$ and $T_0$ have closed (and finite codimensional) image.
\end{lem}
\begin{proof}
Consider the operators $\widetilde T_*,\widetilde T_0:L^2\big([0,1],\C)\to L^2\big([0,1],\C\big)$ defined
respectively by $\widetilde T_*(\mu)=T_*(\mu\cdot Y)$ and $\widetilde T_0(\mu)=T_0(\mu\cdot Y)$.
Clearly, $\mathrm{Im}(\widetilde T^*)\subset\mathrm{Im}(T^*)$ and $\mathrm{Im}(\widetilde T^0)\subset\mathrm{Im}(T^0)$;
it suffices to show that $\widetilde T_*$ and $\widetilde T_0$ have finite codimensional closed image.\footnote{%
Recall that given a closed finite codimensional subspace $X$ of a Hilbert space $H$, then any subspace $Y\subset H$
that contains $X$ is closed (and finite codimensional).}
Now, it is easy to see that both $\widetilde T_*$ and $\widetilde T_0$ are Fredholm operators of index zero;
namely, they are compact perturbations of the isomorphism $L^2\big([0,1],\C\big)\ni\mu\mapsto\mu\cdot g(Y,Y)\in L^2\big([0,1],\C\big)$.
\end{proof}

Keeping in mind \eqref{eq:defAtN}, the adjoint operators $(T_*)^\star$ and $(T_0)^\star$ can be easily computed as:
\begin{multline}
(T_*)^\star(\phi)(t)=\phi(t)\cdot \big(A(t,1)Y(t)\big)-2\big(A(t,1)Y'(t)\big)\cdot\int^1_t\phi(t)\,\de s
\\+2 \big(A(t,1)Y'(t)\big)\cdot \int_0^1
(1-s)\phi(s)\,\de s
\end{multline}
and
\begin{equation}
(T_0)^\star(\phi)(t)=\phi(t)\cdot \big(A(t,1)Y(t)\big)-2\big(A(t,1)Y'(t)\big)\cdot\int^1_t\phi(s)\,\de s.
\end{equation}
Since $T_*$ and $T_0$ have closed image, then also the adjoints $(T_*)^\star$ and $(T_0)^\star$ have closed image, and
\[\mathfrak K_*^\bot=\Ker(T_*)^\perp=\mathrm{Im}\big((T_*)^\star\big),\quad \mathfrak K_0^\bot=\Ker(T_0)^\perp=\mathrm{Im}\big((T_0)^\star\big).\]
The following corollary follows straightforward.\footnote{%
In the sequel, we will use the symbol $A$ to mean $A(\cdot,1)$.}
\begin{cor}
The orthogonal space $\mathfrak K_*^\bot$ in $\mathfrak K$ is:
\begin{equation}
\mathfrak K_*^\bot=\big\{h''\cdot AY+2 h'\cdot AY'\, : h\in H^2\big([0,1];\C\big)\cap H_0^1\big([0,1];\C\big)\big\}
\end{equation}
and the orthogonal space $\mathfrak K_0^\bot$ in $\mathfrak K$ is
\begin{equation}
\mathfrak K_0^\bot=\big\{h'\cdot AY+2 h\cdot AY'\, : h\in H^1\big([0,1];\C\big)\ \text{and}\ h(1)=0 \big\}.
\end{equation}
\end{cor}
\begin{proof}
It follows easily from the preceding observations, keeping in mind that the maps
$H^2\big([0,1];\C\big)\cap H_0^1\big([0,1];\C\big)\ni h\mapsto h''\in L^2\big([0,1],\C)$
and $\big\{h\in H^1\big([0,1];\C\big):h(1)=0\big\}\ni h\mapsto h'\in L^2\big([0,1],\C\big)$ are isomorphisms.
\end{proof}
The bilinear form $I_1$ (defined in \eqref{eq:defIndexformCn}) is represented in $L^2\big([0,1];\C^n\big)$  with respect
to the inner product \eqref{eq:definnerprodK} by the unbounded self-adjoint operator:
\begin{equation}\label{roper}
\mathfrak J(V)=-AV''+ARV
\end{equation}
densely defined on the subspace $D=H^2\big([0,1];\C^n\big)\cap \Hi^\rho(1)$.

By an eigenvalue of the restriction of $I_1$ to $\mathfrak K_*\cap\Hi^\rho(1)$ we will mean a complex number $\lambda_*$ such that
there is a non-zero $V_*\in \mathfrak K_*\cap\Hi^\rho(1)$ satisfying
\begin{equation}
I_1(V_*,W)=\lambda_*\cdot\int_0^1 g_t^{(r)}\big(V_*,W\big)\,\de t
\end{equation}
for every $V,W\in \mathfrak K_*$. Equivalently,  $\lambda_*$ is an eigenvalue of the restriction of $I_1$ to $\mathfrak K_*\cap\Hi^\rho(1)$
if there exists $V_*\in \mathfrak K_*\cap H^2\big([0,1];\C^n\big)\cap \Hi^\rho(1)$ such that
\begin{equation}\label{decom}
\mathfrak J(V_*)-\lambda_*\cdot V_*\in \mathfrak K^\bot_*.
\end{equation}

\begin{prop}\label{difeq*}
A vector $V_*\in  \mathfrak K_*\cap\Hi^\rho(1)$ is an eigenvector for the restriction of $I_1$ to $\mathfrak K_*\cap\Hi^\rho(1)$ with eigenvalue $\lambda_*\in\C$
if and only if $V_*\in H^2\big([0,1];\C^n\big)\cap \Hi^\rho(1)$ and it satisfies
\begin{equation}\label{differeq}
-V''_*+RV_*-\lambda_*\cdot A^{-1}V_*=h''\cdot Y+2 h'\cdot Y',
\end{equation}
where $h$ is the unique map in $H^2\big([0,1];\C\big)\cap H_0^1\big([0,1];\C\big)$ satisfying
\begin{equation}\label{h}
\lambda_*\cdot g^{(r)}_t(V_*,Y)=[h'\cdot g(Y,Y)]'.
\end{equation}
\end{prop}
\begin{proof}
By a boot-strap argument we see that if $V_*$ is an eigenvector then it is differentiable. Using equations
\eqref{roper} and \eqref{decom} we conclude easily that $V_*$ satisfies \eqref{differeq} and \eqref{h} if and only if
$V_*$ is an eigenvector with $\lambda_*$ as eigenvalue. Moreover, these equations imply that $g(V''_*,Y)=g(V_*,Y'')$, so that $V_*\in \mathfrak K_*$.
\end{proof}
We obtain an analogous result for $\mathfrak K_0$.
\begin{prop}\label{thm:eigenrho0}
A vector $V_0\in  \mathfrak K_0\cap\Hi^\rho(1)$ is an eigenvector for the restriction of $I_1$ to $\mathfrak K_0\cap\Hi^\rho(1)$ with eigenvalue $\lambda_0\in\C$ if and only if
$V_0\in H^2\big([0,1];\C^n\big)\cap\Hi^\rho(1)$, $g\big(V'_0(0),Y(0)\big)-g\big(V_0(0),Y'(0)\big)=0$ and the following differential equation is satisfied
\begin{equation}\label{differeq2}
-V''_0+RV_0-\lambda_0\cdot A^{-1}V_0=h'\cdot Y+2 h\cdot Y',
\end{equation}
where
\begin{equation}
h=-\frac{\lambda_0}{g(Y,Y)}\int_t^1 g^{(r)}_s(V_0,Y)\,\de s.
\end{equation}
\end{prop}
\begin{proof}
Similar to Propostion \ref{difeq*}.
\end{proof}
Setting $\lambda_0=0$ in Proposition~\ref{thm:eigenrho0}, one obtains that the elements in the kernel of the restriction of $I_1$ to $\mathfrak K_0\cap\Hi^\rho(1)$
are solutions of the Morse--Sturm system \eqref{eq:defMorseSturm}. This statement is made more precise in the following:
\begin{cor}\label{thm:corstimanu}
$\Ker \big[I_1\big|_{\Hi_0^1(1)\times\Hi_0^1(1)}\big]=\Ker\big[I_1\big|_{\Hi_*^1(1)\times \Hi_*^1(1)}\big]\cap\Hi_0^1(1)$, while
for $\rho\ne1$, $\Ker \big[I_1\big|_{\Hi_0^\rho(1)\times\Hi_0^\rho(1)}\big]=\Ker\big[I_1\big|_{\Hi_*^\rho(1)\times \Hi_*^\rho(1)}\big]$.
In particular:
\begin{equation}\label{eq:corstimanu}
\nu_0(1,1)\le\nu_*(1,1)\le\nu_0(1,1)+1,\quad \nu_0(\rho,1)=\nu_*(\rho,1)\ \text{for $\rho\ne1$}.
\end{equation}
\end{cor}
\begin{proof}
For $\lambda_0=0$, equation \eqref{differeq2} is the Morse--Sturm system \eqref{eq:defMorseSturm}; the conclusion follows
easily from Corollary~\ref{thm:corker0ker*}.
The second inequality in \eqref{eq:corstimanu} follows from the fact that $\Hi_0^\rho(1)$ has codimension
$1$ in $\Hi_*^\rho(1)$.
\end{proof}
\begin{cor}
If $\mathcal A\subset\mathds S^1\setminus\{1\}$ is a connected subset that does not contain elements in the spectrum of $\mathfrak P$, then
the map $\lambda_0(\rho,1)$ is constant on $\mathcal A$.
\end{cor}
\begin{proof}
The assumption is that $\nu_*(\rho,1)$ vanishes on $\mathcal A$, which by Corollary~\ref{thm:corstimanu}
implies that also $\nu_0(\rho,1)$ vanishes on $\mathcal A$. Corollary~\ref{thm:contindex} concludes the thesis.
\end{proof}
\end{section}

\begin{section}{On the index sequences}
\label{sec:sequence}
\subsection{A Fourier theorem}
In the following, given $\rho\in\mathds S^1$ and $w\in\mathcal H_*^\rho(1)$,
it will be useful to consider the (continuous) extension $w:\R\to\C^n$ of
$w$ defined by:
\begin{equation}\label{eq:extension}
\phantom{\qquad \forall\,t\in\left[0,1\right[,\ \forall\,N\in\Z.}w(t+N)=\rho^NT^{-N}w(t),\qquad \forall\,t\in\left[0,1\right[,\ \forall\,N\in\Z.
\end{equation}
Observe that such extension does not satisfy $g(w',Y)-g(w,Y')=\text{const.}$ in $\R$, unless $\rho=1$ or
$w\in\mathcal H_0^\rho(1)$.
\begin{prop}
For all $N\ge1$, set $\omega=e^{2\pi i/N}$ and define the map
\[\Psi_N:\mathcal H_*^1(N)\to\Big[\bigoplus\nolimits_{k=1}^{N-1}\mathcal H_0^{\omega^k}(1)\Big]\oplus \mathcal H_*^1(1)\]
by:
\[V\longmapsto (V_1,\ldots,V_N)\]
where \begin{equation}\label{eq:defVk}
V_k(t)=\frac1N\sum_{j=0}^{N-1}\omega^{-kj}T^{j}V\left(\frac{t+j}N\right),\end{equation}
for all $t\in\left[0,1\right[$.
Then, $\Psi_N$ is a linear isomorphism, whose inverse
\[\Upsilon_N:\Big[\bigoplus\nolimits_{k=1}^{N-1}\mathcal H_0^{\omega^k}(1)\Big]\oplus \mathcal H_*^1(1)\longrightarrow \mathcal H_*^1(N)\]
is given by $(V_1,\ldots,V_N)\mapsto V$,
where\footnote{In equation \eqref{eq:inversePsiN} we are assuming that all the $V_k$'s have been extended to $\R$
as in \eqref{eq:extension}. An immediate calculation  shows that if $V$ and the $V_k$'s are extended to $\R$
according to \eqref{eq:extension}, then equality \eqref{eq:defVk} holds for all $t\in\R$.}
\begin{equation}\label{eq:inversePsiN}
V(t)=\sum_{k=1}^NV_k(tN),\end{equation} for all $t\in[0,1]$. The isomorphism $\Psi_N$ carries $\mathcal H^1_0(N)$ onto
the direct sum $\bigoplus\limits_{k=1}^N\mathcal H_0^{\omega^k}(1)$.
\end{prop}
\begin{proof}
The proof is a matter of direct calculations, based on a repeated
use of the identity:
\[\sum_{r=0}^{N-1}\omega^{sr}=\begin{cases}0,&\text{if $s\not\equiv0\mod N$;}\\ N,&\text{if $s\equiv0\mod N$.}\end{cases}\]
The details of the computations are as follows.
Clearly, $\Psi_N$ is linear and bounded. Let us show that it is well defined, i.e.,
that for all $k=1,\ldots,N-1$, the map $V_k$ is in $\mathcal H_0^{\omega^k}(1)$ and that $V_N\in \mathcal H_*^{\omega^k}(1)$. We compute:
\begin{multline*}\omega^{-k}TV_k(1)=\frac1N\sum_{j=0}^{N-1}\omega^{-k(j+1)}T^{j+1}V(\tfrac{j+1}N)=
\frac1N\sum_{j=1}^{N}\omega^{-kj}T^jV(j/N)\\ =\frac1N\left(\sum_{j=1}^{N-1}\omega^{-kj}T^jV(j/N)+T^NV(1)\right)=
 \frac1N\left(\sum_{j=1}^{N-1}\omega^{-kj}T^jV(j/N)+V(0)\right)\\=\frac1N\sum_{j=0}^{N-1}\omega^{-kj}T^{j}V(j/N)=V_k(0).
\end{multline*}
Moreover, for $t\in\left[0,1\right[$, setting $s_j=(t+j)/N$, $j=0,\ldots,N-1$
and recalling formula \eqref{eq:periodYNeRN}:
\begin{multline*}
g\big(V_k'(t),Y(t)\big)-g\big(V_k(t),Y'(t)\big)\\=\frac1N\sum_{j=0}^{N-1}\omega^{-kj}\left[\tfrac1Ng\big(T^{j}V'(\tfrac{t+j}N),Y(t)\big)
-g\big(T^{j}V(\tfrac{t+j}N),Y'(t)\big)\right]\\=
\frac1N\sum_{j=0}^{N-1}\omega^{-kj}\left[\tfrac1Ng\big(T^{j}V'(s_j),Y_N(s_j-\tfrac{j}N)\big)
-\tfrac1Ng\big(T^{j}V(s_j),Y_N'(s_j-\tfrac{j}N)\big)\right]
\\=
\frac1{N^2}\sum_{j=0}^{N-1}\omega^{-kj}\left[g\big(T^{j}V'(s_j),T^{j}Y_N(s_j)\big)
-g\big(T^{j}V(s_j),T^{j}Y_N'(s_j)\big)\right]\\
=
\frac1{N^2}\sum_{j=0}^{N-1}\omega^{-kj}\left[g\big(V'(s_j),Y_N(s_j)\big)
-g\big(V(s_j),Y_N'(s_j)\big)\right]\\=\frac{C_V}{N^2}\sum_{j=0}^{N-1}\omega^{-kj}=\begin{cases}0,&\text{if $k<N$;}\\ \frac1N{C_V},&\text{if $k=N$.}\end{cases}
\end{multline*}
This proves that $V_k\in\mathcal H_0^{\omega^k}(1)$ for $k<N$ and that $V_N\in\mathcal H_*^1(1)$. In order to conclude the proof it
remains to verify that \eqref{eq:inversePsiN} defines an inverse for $\Psi_N$.
This is also a straightforward calculation. Given $V\in\mathcal H_*^1(N)$, then:
\[\frac1N\sum_{k=1}^N\sum_{j=0}^{N-1}\omega^{-kj}T^jV\big(\tfrac{tN+j}N\big)=\frac1N\sum_{j=0}^{N-1}T^jV\big(\tfrac{tN+j}N\big)\left(\sum_{k=1}^N\omega^{-kj}\right)=
V(t),\]
which proves that $\Upsilon_N\circ\Psi_N$ is the identity.

Conversely, given $(V_1,\ldots,V_N)\in \Big[\bigoplus\nolimits_{k=1}^{N-1}\mathcal H_0^{\omega^k}(1)\Big]\oplus \mathcal H_*^1(1)$,
$k\in\{1,\ldots,N\}$ and $t\in\left[0,1\right[$:
\begin{multline*}
\frac1N\sum_{j=0}^{N-1}\omega^{-kj}T^j\sum_{l=1}^NV_l(t+j)=\frac1N\sum_{j=0}^{N-1}\omega^{-kj}\sum_{l=1}^N\omega^{lj}V_l(t)
=\frac1N\sum_{j=0}^{N-1}\sum_{l=1}^N\omega^{j(l-k)}V_l(t)\\=\frac1N\sum_{l=1}^NV_l(t)\left[\sum_{j=0}^{N-1}\omega^{j(l-k)}\right]=V_k(t),
\end{multline*}
which proves that $\Psi_N\circ \Upsilon_N$ is the identity. This concludes the proof.
\end{proof}
Finally, the desired result on the index sequences:
\begin{prop}[Fourier theorem]\label{thm:prop4.2}
For all $N\ge1$, the following identities hold:
\begin{eqnarray*}
&\lambda_*(1,N)=\lambda_*(1,1)+\sum_{k=1}^{N-1}\lambda_0(\omega^k,1),\\
&\lambda_0(1,N)=\sum_{k=1}^N\lambda_0(\omega^k,1),
\end{eqnarray*}
where $\omega=e^{2\pi i/N}$.
\end{prop}
\begin{proof}
The result is obtained by showing that, given $V_k,W_k\in\mathcal H_0^{\omega^k}(1)$,
$k=1,\ldots,N-1$, $V_N,W_N\in\mathcal H_*^1(1)$,
and setting $V=\Psi_N^{-1}(V_1,\ldots,V_N)$, $W=\Psi_N^{-1}(W_1,\ldots,W_N)$,
the following identity holds:
\[I_N(V,W)=\sum_{k=1}^NI_1(V_k,W_k).\]
This is obtained by a direct calculation, keeping in mind that:
\begin{itemize}
\item $V_k(s+l-1)=\omega^{k(l-1)}T^{1-l}V_k(s)$, $V_k'(s+l-1)=\omega^{k(l-1)}T^{1-l}V_k'(s)$;
\item $W_k(s+l-1)=\omega^{k(l-1)}T^{1-l}W_k(s)$, $W_k'(s+l-1)=\omega^{k(l-1)}T^{1-l}W_k'(s)$;
\item $R(s+l-1)=T^{1-l}R(s)T^{l-1}$,
\end{itemize}
 for all $s\in[0,1]$. Then:
 \begin{multline*}
 I_N(V,W)=N^2\sum_{k,r=1}^N\int_0^1\Big[g\big(V_k'(tN),W_r'(tN)\big)+g\big(R(tN)V_k(tN),W_r(tN)\big)\Big]\,\mathrm dt
\displaybreak[0]\\=N^2\sum_{k,r=1}^N\sum_{l=1}^N\int_{\frac{l-1}N}^{\frac lN}\Big[g\big(V_k'(tN),W_r'(tN)\big)+g\big(R(tN)V_k(tN),W_r(tN)\big)\Big]\,\mathrm dt
\displaybreak[0]\\
= N\sum_{k,r=1}^N\sum_{l=1}^N\int_{l-1}^{l}\Big[g\big(V_k'(s),W_r'(s)\big)+g\big(R(s)V_k(s),W_r(s)\big)\Big]\,\mathrm ds
\displaybreak[0]\\
=\! N\!\!\sum_{k,r,l=1}^N\!\int_{0}^{1}\!\!\Big[g\big(V_k'(s+l-1),W_r'(s+l-1)\big)+g\big(R(s+l-1)V_k(s+l-1),W_r(s+l-1)\big)\!\Big]\mathrm ds
\displaybreak[0]\\
= N\sum_{k,r=1}^N\sum_{l=1}^N\omega^{(k-r)(l-1)}\int_{0}^{1}\Big[g\big(V_k'(s),W_r'(s)\big)+g\big(R(s)V_k(s),W_r(s)\big)\Big]\,\mathrm ds
\displaybreak[0]\\
=\sum_{k=1}^N\int_0^1\Big[g\big(V_k'(s),W_k'(s)\big)+g\big(R(s)V_k(s),W_k(s)\big)\Big]\,\mathrm ds=\sum_{k=1}^NI_1(V_k,W_k),
 \end{multline*}
 which concludes the proof.
\end{proof}
\subsection{On the jumps of the index function}\label{sub:jumpsindexfnctn}
The question of determining the value of the jumps of the index function $\rho\mapsto\lambda_0(\rho,1)$
is rather involved, and it will be treated only marginally in this subsection. Let us start with a simple
observation on the index of continuously varying essentially positive symmetric bilinear forms, whose
proof is omitted:
\begin{lem}\label{thm:lem4.3}
Let $B:\mathfrak H\times\mathfrak H\to\C$ be a Fredholm Hermitian form on the complex Hilbert space
$\mathfrak H$, and let $\{\mathcal D^t\}_{t\in I}$ be a continuous family of closed subspaces of
$\mathfrak H$ such that the restriction $B_t$ of $B$ to $\mathcal D^t\times\mathcal D^t$ is essentially
positive for all $t\in I$. If $t_0$ is an isolated instant in the interior of $I$ such that
$B_{t_0}$ is nondegenerate, then for $\varepsilon>0$ small
enough:
\[\left\vert\mathrm n_-\big(B_{t_0+\varepsilon}\big)-\mathrm n_-\big(B_{t_0-\varepsilon}\big)\right\vert
\le \Dim\big(\Ker(B_{t_0})\big).\qed\]
\end{lem}
\begin{cor}\label{thm:cor4.4}
Let $e^{2\pi i\widetilde\theta}\in\mathds S^1\setminus\{1\}$ be a discontinuity point for the map $\rho\mapsto\lambda_0(\rho,1)$.
Then:
\[\left\vert\lim_{\theta\to0^+}\big[\lambda_0\big(e^{2\pi i(\widetilde\theta+\theta)},1\big)-\lambda_0\big(e^{2\pi i(\widetilde\theta-\theta)},1\big)\big]\right\vert\le
\nu_0\big(e^{2\pi i\widetilde\theta},1\big).\]
If $Y$ is not a singular solution of \eqref{eq:defMorseSturm}, the same conclusion holds also for $\widetilde\theta=0$.
\end{cor}
\begin{proof}
It follows immediately from Proposition~\ref{thm:corcriteriosmoothness} (or Corollary~\ref{thm:cor2.9} in the singular case),
Proposition~\ref{thm:prop2.10} and Lemma~\ref{thm:lem4.3}.
\end{proof}

Under a certain nondegeneracy assumption, the jump of the index function at a discontinuity point
can be computed in terms of a finite dimensional reduction (compare with \cite[Theorem~IV, p.\ 180]{Bo4}).
This finite dimensional reduction is rather technical, and we will only sketch its construction.
Given $e^{2\pi i\widetilde\theta}\in\mathfrak s(\mathfrak P_0)\cap\mathds S^1$, let us define a Hermitian form
$B_{\widetilde\theta}$ on the finite dimensional vector space $N_{\widetilde\theta}=\Ker\big(\mathfrak P_0-e^{2\pi i\widetilde\theta}\big)$ as follows.
Identify vectors $v\in N_{\widetilde\theta}$ with functions $V\in\mathcal H_0^{e^{2\pi i\widetilde\theta}}(1)$ in the kernel
of $I_1$ (use Proposition~\ref{thm:kerneleigenspace} and Corollary~\ref{thm:corstimanu}) and, given one such $V$, choose
an arbitrary $C^1$-map $\mathfrak V:\left]-\varepsilon,\varepsilon\right[\to \mathcal H$ with $\mathfrak V(s)\in\mathcal H_0^{e^{2\pi i(\widetilde\theta+s)}}(1)$
for all $s$,
and such that $\mathfrak V(0)=V$. Finally, denote by $P_{\widetilde\theta}:\mathcal H\to \mathcal H_0^{e^{2\pi i\widetilde\theta}}(1)$ the
orthogonal projection, and set:
\[B_{\widetilde\theta}(v,w)=I_1\big(P_{\widetilde\theta}\,\mathfrak V'(0),\mathfrak W(0)\big)+I_1\big(\mathfrak V(0),P_{\widetilde\theta}\,\mathfrak W'(0)\big);\]
it is not hard to show that $B_{\widetilde\theta}$ is well defined, i.e., that the right hand side in the above formula
does not depend on the choice of the $C^1$-maps $\mathfrak V$ and $\mathfrak W$.
\begin{prop}\label{thm:compjumpnondeg}
Let $e^{2\pi i\widetilde\theta}\in\mathds S^1$ be a discontinuity point for $\rho\mapsto\lambda_0(\rho,1)$
(with $e^{2\pi i\widetilde\theta}\ne1$ if $Y$ is a singular solution of \eqref{eq:defMorseSturm}). Then, if $B_{\widetilde\theta}$
is nondegenerate, for $\theta>0$ small enough:
\[\lambda_0\big(e^{2\pi i(\widetilde\theta+\theta)},1\big)-\lambda_0\big(e^{2\pi i(\widetilde\theta-\theta)},1\big)=-\mathrm{signature}(B_{\widetilde\theta}).\]
\end{prop}
It is clear from our construction that the value of the jump of $\lambda_0(\rho,1)$   at a discontinuity $\rho_0$
can be computed as an algebraic count of the eigenvalues through zero of the path $\rho\mapsto\mathcal T_\rho$  of self-adjoint Fredholm
operators representing the index form $I_1$ in $\mathcal H_0^\rho(1)$ as $\rho$ runs
in the arc $\mathcal A_{\rho_0}=\{e^{2\pi i\theta}\rho,\ \theta\in[-\varepsilon,\varepsilon]\}$. Technically speaking,
this is the so-called \emph{spectral flow} of the path $\mathcal A_{\rho_0}\ni\rho\mapsto \mathcal T_\rho$, see for
instance \cite{Phillips} for details on the spectral flow. By  Proposition~\ref{thm:corcriteriosmoothness}
(or Proposition~\ref{thm:hoanalytic} in the singular case), the path $\mathcal T_\rho$ is analytic, in which case
higher order methods are available in order to compute the value of the spectral flow at each
degeneracy instant (see \cite{GPP} for details).
The result in Proposition~\ref{thm:compjumpnondeg} is a special case of this method.

\end{section}

\begin{section}{Closed geodesics in stationary Lorentzian manifolds}
\label{sec:closedgeodesics}

\subsection{Closed geodesics}\label{sub:closedgeodesics}
Let us consider a Lorentzian manifold $(M,\mathfrak g)$, with $\mathrm{dim}(M)=n+1$, endowed with a timelike Killing vector field
$\mathcal K$; let $\nabla$ be the Levi--Civita connection of $\mathfrak g$ and $\mathcal R(X,Z)=[\nabla_X,\nabla_Z]-\nabla_{[X,Z]}$
its curvature tensor (see \cite{BEE,HaEl73,One83} for details).
An auxiliary Riemannian metric $\mathfrak g^{\mathrm R}$ on $M$ is obtained by taking $\mathfrak g^{\mathrm R}(v,w)=\mathfrak g(v,w)-
2\mathfrak g(v,\mathcal K)\,\mathfrak g(w,\mathcal K)\mathfrak g(\mathcal K,\mathcal K)^{-1}$.
Let $\gamma:[0,1]\to M$ be a non constant closed geodesic in $(M,\mathfrak g)$ and denote by $\widetilde\gamma:\R\to M$ its
periodic extension to the real line; for $N\ge1$, let $\gamma^N:[0,1]\to M$
denote the $N$-th iterate of $\gamma$, defined by $\gamma^N(s)=\widetilde\gamma(Ns)$, $s\in[0,1]$.
There are two constants associated to $\gamma$: $E_\gamma=\mathfrak g(\dot\gamma,\dot\gamma)$ and $c_\gamma=\mathfrak g(\dot\gamma,\mathcal Y)$;
observe that, by causality, $\gamma$ is spacelike, and thus $E_\gamma>0$. Define a smooth vector field $\mathcal Y$ along
$\gamma$ by setting $\mathcal Y(s)=\mathcal K\big(\gamma(s)\big)-c_\gamma E_\gamma^{-1}\dot\gamma(s)$; this is a periodic timelike
Jacobi field along $\gamma$ which is everywhere orthogonal to $\dot\gamma$.
The index form $\mathcal I_\gamma$ of $\gamma$, which is the second variation of the geodesic action functional
at $\gamma$, is given by:
\[\mathcal I_\gamma(\mathcal V,\mathcal W)=\int_0^1\big[\mathfrak g\big(\Dds \mathcal V,\Dds \mathcal W\big)+\mathfrak g\big(\mathcal R(\dot\gamma,\mathcal V)\,\dot\gamma,\mathcal W)\big]\,\mathrm ds;\]
here $\Dds$ denotes covariant differentiation along $\gamma$.
$\mathcal I_\gamma$ is a bounded symmetric bilinear form defined on the real Hilbert space $\mathcal H^\gamma$ of all periodic vector fields
along $\gamma$ of Sobolev class $H^1$ that are \emph{everywhere orthogonal} to $\dot\gamma$. Consider the following closed
subspaces of $\mathcal H^\gamma$:
\begin{align*}
&\mathcal H^\gamma_*=\big\{\mathcal V\in\mathcal H^\gamma:\mathfrak g\big(\Dds \mathcal V,\mathcal Y\big)-\mathfrak g(\mathcal V,\Dds\mathcal Y)\ \text{is constant on}\ [0,1]\big\},
\\
&\mathcal H^\gamma_0=\big\{\mathcal V\in\mathcal H^\gamma:\mathfrak g\big(\Dds \mathcal V,\mathcal Y\big)-\mathfrak g(\mathcal V,\Dds\mathcal Y)=0\ \text{a.e.\ on}\ [0,1]\big\}.
\end{align*}
Elements in $\mathcal H^\gamma_*$ are variational vector fields along $\gamma$ corresponding to variations of
$\gamma$ by curves $\mu$ for which the quantity $\mathfrak g(\dot\mu,\mathcal Y)$ is constant; similarly, elements
of $\mathcal H^\gamma_0$ correspond to variations of $\gamma$ by curves $\mu$ for which $\mathfrak g(\dot\mu,\mathcal Y)$
is equal to the constant $c_\gamma$.
The restrictions of $I_\gamma$ to $\mathcal H^\gamma_*$ and to $\mathcal H^\gamma_0$ have finite index
(see \cite{BilMerPic} for details); they will be denoted respectively by $\mu(\gamma)$ and $\mu_0(\gamma)$, and called
the \emph{Morse index} and the \emph{restricted Morse index} of $\gamma$. Since $\mathcal H^\gamma_0$ has codimension
$1$ in $\mathcal H^\gamma_*$, then $\mu_0(\gamma)\le\mu(\gamma)\le\mu_0(\gamma)+1$; we will denote by
$\epsilon_\gamma\in\{0,1\}$ the difference $\mu(\gamma)-\mu_0(\gamma)$. The \emph{nullity} of
$\gamma$, $\mathrm n(\gamma)$ is defined as the dimension of the space of periodic Jacobi fields
along $\gamma$ that are everywhere orthogonal to $\dot\gamma$, or, equivalently, as the dimension
of the kernel of  $I_\gamma$ in $\mathcal H^\gamma$. Similarly, we will denote by $\mathrm n_0(\gamma)$ the
\emph{restricted nullity} of $\gamma$, defined as the dimension of the kernel of the restriction of $I_\gamma$ to
$\mathcal H^\gamma_0$.
Let $\Sigma$ be a hypersurface of $M$ through $\gamma(0)$ which is orthogonal
to $\dot\gamma(0)$; denote by $TM_{E_\gamma}\big\vert_\Sigma$ the restriction to $\Sigma$  of the sphere
bundle $\{v\in TM: g(v,v)=E_\gamma\}$. Let $\mathcal P_\Sigma:\mathcal U_\Sigma\to\mathcal U_\Sigma$ denote the Poincar\'e map of $\Sigma$,
defined in a sufficiently small neighborhood $\mathcal U_\Sigma$ of $\dot\gamma(0)$ in $TM_{E_\gamma}\big\vert_\Sigma$.
Recall that $\mathcal P_\Sigma$ preserves the symplectic structure inherited from $TM$
(here one uses the metric $\mathfrak g$ to induce a symplectic form from $TM^*$ to $TM$), and that $\dot\gamma(0)$ is a
fixed point of $\mathcal P_\Sigma$. The \emph{linearized Poincar\'e map of $\gamma$} is the differential
$\mathfrak P_\gamma=\mathrm d\mathcal P_\Sigma\big(\dot\gamma(0)\big):T_{\dot\gamma(0)}\mathcal U_\Sigma\to T_{\dot\gamma(0)}\mathcal U_\Sigma$.
If one uses the horizontal distribution of the connection $\nabla$ to identify $T_{\dot\gamma(0)}(TM)$ with the
direct sum $T_{\gamma(0)}M\oplus T_{\gamma(0)}M$, then $T_{\dot\gamma(0)}\mathcal U_\Sigma$ is identified
with $\dot\gamma(0)^\perp\oplus\dot\gamma(0)^\perp$, and $\mathfrak P_\gamma(v,w)=\big(J(1),\Dds J(1)\big)$, where
$J$ is the unique Jacobi field along $\gamma$ satisfying the initial conditions $J(0)=v$ and $\Dds J(0)=w$.
The closed geodesic $\gamma$ will be called \emph{singular} if the covariant derivative $\Dds\mathcal K$
of the restriction of the Killing field $\mathcal K$ along $\gamma$ is pointwise multiple
of the orthogonal projection of $\mathcal K$ onto $\dot\gamma^\perp$.
This condition is the same as assuming that the covariant derivative $\Dds\mathcal Y$
of the Jacobi field $\mathcal Y$ is pointwise multiple of $\mathcal Y$.
Observe that, by Lemma~\ref{thm:spazicoincsing}, if $\gamma$ is singular then $\epsilon_\gamma=0$.
\smallskip

We will consider the complexification of the Hilbert spaces defined above, as well as the complexification of the linear maps
and the sesquilinear extension of $\mathcal I_\gamma$;
these complexified objects  will be denoted by the same symbols as their real counterparts.
\subsection{Geodesics and Morse--Sturm systems}\label{sub:geomorsesturm}
For $t\in[0,1]$, denote by $\mathbf P_t:T_{\gamma(0)}M\to T_{\gamma(t)}M$ the parallel transport; observe
that $\mathbf P_t$ carries $\dot\gamma(0)^\perp$ isomorphically onto $\dot\gamma(t)^\perp$. We choose
an isomorphism $\phi_0:\R^n\to\dot\gamma(0)^\perp$, and we denote by $\phi_t:\R^n\to\dot\gamma(t)^\perp$ the
isomorphism $\mathbf P_t\circ\phi_0$. Finally, set $T:\R^n\stackrel\cong\longrightarrow\R^n$, $T=\phi_0^{-1}\circ\mathbf P_1\circ\phi_0=\phi_0^{-1}\phi_1$.
Consider the following data to build up a Morse--Sturm system as described in Section~\ref{sec:morsesturm}.
Let $g$ be the nondegenerate symmetric bilinear form on $\R^n$ given by the pull-back $\phi_0^*\mathfrak g$;
since the parallel transport is an isometry, then $g=\phi_t^*\mathfrak g$ for all $t\in[0,1]$, and $T$ is $g$-preserving.\footnote{%
It is interesting to observe here that $T$ belongs to the connected component of the identity of $\mathrm O(\R^n,g)$ precisely
when the geodesic $\gamma$ is orientation preserving, i.e., when the parallel transport $\mathbf P_1$ is
orientation preserving.}
Define $R(t):\R^n\to\R^n$ by $R(t)=\phi_t^{-1}\circ\big[\mathcal R\big(\dot\gamma(t),\cdot\big)\,\gamma(t)\big]\circ\phi_t$;
since $\mathcal R$ is $\mathfrak g$-symmetric, then $R$ is $g$-symmetric. Since $\gamma$ is periodic, $\mathcal R\big(\dot\gamma(0),\cdot\big)\,\dot\gamma(0)=
\mathcal R\big(\dot\gamma(1),\cdot\big)\,\dot\gamma(1)$, and thus $T^{-1}R(0)T=R(1)$.
Again, we will consider complexifications of these objects, that will be denoted by the same symbols as their real counterparts.

Using the isomorphisms $(\phi_s)_{s\in[0,1]}$, from a map $V:[0,1]\to\C^n$ one obtains a vector field $\mathcal V$ along $\gamma$ which is
orthogonal to $\dot\gamma$, defined by $\mathcal V(s)=\phi_s\big(V(s)\big)$; the periodicity
condition $\mathcal V(0)=\mathcal V(1)$ corresponds to the condition $TV(1)=V(0)$. An immediate computation shows that
the map $V\mapsto\mathcal V$, denoted by $\Psi$, carries the space of solutions of the Morse--Sturm system \eqref{eq:defMorseSturm} to the space of Jacobi fields along
$\gamma$ that are everywhere orthogonal to $\dot\gamma$; define $Y$ as $\Psi^{-1}(\mathcal Y)$, where
$\mathcal Y$ is the orthogonal timelike Jacobi field along $\gamma$ defined above, so that $Y$ satisfies (e) in Subsection~\ref{sub:basicsetup}.
It is also immediate to see that $\gamma$ is a singular closed geodesic as defined in Subsection~\ref{sub:closedgeodesics}
exactly when $Y$ is a singular solution of the Morse--Sturm system \eqref{eq:defMorseSturm}.
\begin{example}\label{exa:casostatico}
An important class of examples of singular closed geodesics can be obtained by considering
\emph{static} Lorentzian manifolds $(M,\mathfrak g)$, i.e., Lorentzian manifolds admitting a timelike Killing vector field $\mathcal K$
whose orthogonal distribution $\smash{\mathcal K}^\perp$ is integrable.
Every integral submanifold of $\smash{\mathcal K}^\perp$ is a totally geodesic
submanifold of $M$; every closed geodesic in $M$ which is orthogonal to $\mathcal K$
at some point is contained in one such integral submanifold. Moreover, if $(M,\mathfrak g)$ is globally hyperbolic
or if $M$ is simply connected, then every closed geodesic in $(M,\mathfrak g)$ is orthogonal to $\mathcal K$
and therefore contained  in an integral submanifold $\Sigma$ of $\smash{\mathcal K}^\perp$.
Every such geodesic is singular. Namely, let $\{E_i(t)\}_i$ be a parallel frame of $T\Sigma$ along $\gamma$ relatively
to the Riemannian metric on $\Sigma$ obtained by restriction of $\mathfrak g$.
Since $\Sigma$ is totally geodesic, then $E_i$ is a parallel also in $(M,\mathfrak g)$;
and since $\mathfrak g(\mathcal K,E_i)=0$, by differentiating we obtain
$\mathfrak g\big(\Dds\mathcal K,E_i\big)=0$ for all $i$. This shows that
$ \Dds\mathcal K$ is pointwise multiple of $\mathcal K$, i.e.,
$\gamma$ is singular. It follows in particular that $\epsilon_\gamma=0$ for
all closed geodesic $\gamma$, more generally, the same conclusion holds when $\gamma$ is contained in a totally geodesic hypersurface
of $M$ which is everywhere orthogonal to $\mathcal K$.
Closed geodesics in compact static Lorentzian manifolds are studied in \cite{San06}.
\end{example}

Using the isomorphism $\phi_0\oplus\phi_0$, the restriction of the linearized Poincar\'e map $\mathfrak P_\gamma$ to
$\dot\gamma(0)^\perp\oplus\dot\gamma(0)^\perp$ is identified with the linear Poincar\'e map $\mathfrak P$ of the
Morse--Sturm system defined in Subsection~\ref{sub:linearpoicare}. The restricted Poincar\'e map $\mathfrak P_0$
of the Morse--Sturm system correspond to the restriction of $\mathfrak P_\gamma$ to the invariant
subspace $\mathcal E_0\subset\dot\gamma(0)^\perp\oplus\dot\gamma(0)^\perp$ consisting of pairs $(v,w)$ such that
$g\big(w,\mathcal Y(\gamma(0))\big)-g(v,\Dds\mathcal Y(\gamma(0))\big)=0$.

Recalling the notations in Subsection~\ref{sub:closedsubspaces}, we have
an isomorphism $\Psi:\mathcal H^1(1)\to\mathcal H^\gamma$ that carries the spaces $\mathcal H_*^1(1)$ and $\mathcal H_0^1(1)$ respectively
onto $\mathcal H^\gamma_*$ and $\mathcal H^\gamma_0$. Moreover, the pull-back by $\Psi$ of the index form $I_\gamma$ is
the bilinear form $I_1$ defined in \eqref{eq:defIndexformCn}.
In total analogy, the index form $I_{\gamma^N}$ of the $N$-th iterate $\gamma^N$ of $\gamma$ corresponds
to the index form $I_N$ in \eqref{eq:defIndexformCn}.
The indexes and the nullities of the geodesic $\gamma$ are therefore related to the indexes and nullities of the
Morse--Sturm system (Subsection~\ref{sub:nulindseq}) by:
\[\mu(\gamma^N)=\lambda_*(1,N),\quad\mu_0(\gamma^N)=\lambda_0(1,N),\quad \mathrm n(\gamma^N)=\nu_*(1,N),\quad\!\!\mathrm n_0(\gamma^N)=\nu_0(1,N).\]
If we define $\Lambda_\gamma, N_\gamma:\mathds S^1\to\N$ by setting:
\[\Lambda_\gamma(\rho)=\lambda_0(\rho,1),\quad N_\gamma(\rho)=\nu_0(\rho,1),\qquad\rho\in\mathds S^1,\]
we can state the central result of the paper:
\begin{teo}\label{thm:iterazioneindice} For all $N\ge1$, the following statements hold:
\begin{enumerate}
\item\label{itm:iterazioneindice1}
$\mu(\gamma^N)=\mu(\gamma)+\sum\limits_{k=1}^{N-1}\Lambda_\gamma\big(e^{2\pi ik/N}\big)=\epsilon_\gamma+\sum\limits_{k=1}^N\Lambda_\gamma\big(e^{2\pi ik/N}\big)$.
\item\label{itm:iterazioneindice1bis} $\mu_0(\gamma^N)=\sum\limits_{k=1}^N\Lambda_\gamma\big(e^{2\pi ik/N}\big)$.
\item\label{itm:iterazioneindice2} If $\gamma$ is non singular, the jumps of $\Lambda_\gamma$ can only occur at points of the spectrum of
$\mathfrak P_\gamma$ that lie on the unit circle; if $\gamma$ is singular a jump of $\Lambda_\gamma$ can occur also at $1$.
\item\label{itm:iterazioneindice3} If $\gamma$ is non singular and $e^{2\pi i\widetilde\theta}$ is a discontinuity point of $\Lambda_\gamma$, then
$\Lambda_\gamma\big(e^{2\pi i\widetilde\theta}\big)\le\lim\limits_{\theta\to0^\pm}\Lambda_\gamma\big(e^{2\pi i\widetilde\theta+\theta}\big)$.
Moreover, the following estimate on the jump of $\Lambda_\gamma$ holds:
\[\left\vert\lim_{\theta\to0^+}\Lambda_\gamma\big(e^{2\pi i\widetilde\theta+\theta}\big)-\lim_{\theta\to0^-}\Lambda_\gamma\big(e^{2\pi i\widetilde\theta+\theta}\big)
\right\vert\le N_\gamma\big(e^{2\pi i\widetilde\theta}\big).\]
The same conclusion holds when $\gamma$ is singular, under the additional assumption that $e^{2\pi i\widetilde\theta}\ne1$.
\item\label{itm:correctiontermiterated} For all $N\ge1$, $\epsilon_\gamma=\epsilon_{\gamma^N}$.
\end{enumerate}
\end{teo}
\begin{proof}
Parts \eqref{itm:iterazioneindice1} and \eqref{itm:iterazioneindice1bis} follow from Proposition~\ref{thm:prop4.2}.
Part~\eqref{itm:iterazioneindice2} follows from Corollary~\ref{thm:contindex} and Proposition~\ref{thm:kerneleigenspace}.
The first statement in \eqref{itm:iterazioneindice3} is the lower semi-continuity property of the function
$\rho\mapsto\lambda_0(\rho,1)$ proved in (a) of Proposition~\ref{thm:prop2.15} and the second one  follows from Corollary~\ref{thm:cor4.4};
part~\eqref{itm:correctiontermiterated} follows from \eqref{itm:iterazioneindice1} and \eqref{itm:iterazioneindice1bis}.
\end{proof}

\subsection{Iteration formulas for the Morse index}
Let us show that the sequence $\mu(\gamma^N)$ has linear growth in $N$:
\begin{prop}\label{thm:stimacrescitalineare}
Either $\mu(\gamma^N)$ is a constant sequence (equal to $\epsilon_\gamma$), or the limit:
\[\lim_{N\to\infty}\tfrac1N\,\mu\big(\gamma^N\big)=\lim_{N\to\infty}\tfrac1N\,\mu_0\big(\gamma^N\big)\]
exists, it is finite and positive. In this case, its value is given by a sum of the type:
\begin{equation}\label{eq:formulasomma} a_0+\sum_{j=1}^Ka_j\theta_j,\end{equation}
where $a_i$ are integers, $a_0>0$, and $0\le\theta_1<\theta_2<\ldots<\theta_K<1$ are real numbers such that
$e^{2\pi i\theta_j}$ belong to the spectrum of $\mathfrak P_\gamma$.
\end{prop}
\begin{proof}
By part~\eqref{itm:iterazioneindice2} of Theorem~\ref{thm:iterazioneindice}, $\Lambda_\gamma$ has a finite number of discontinuities,
thus it is Riemann integrable  and  the limit:
\[\lim_{N\to\infty}\tfrac1N\,\mu\big(\gamma^N\big)=\lim_{N\to\infty}\frac1N\,\sum\limits_{k=1}^N\Lambda_\gamma\big(e^{2\pi ik/N}\big)\]
equals the integral $\displaystyle\int_{\mathds S^1}\Lambda_\gamma\,\mathrm d\rho\ge0$. Using \eqref{itm:iterazioneindice1}
in Theorem~\ref{thm:iterazioneindice}, $\mu(\gamma^N)$ is constant if and only if
$\Lambda_\gamma$ vanishes identically (recall that $\Lambda_\gamma$ is piecewise constant), hence if $\mu(\gamma^N)$ is not constant,
$\int_{\mathds S^1}\Lambda_\gamma\,\mathrm d\rho>0$. If $e^{2\pi i\theta_j}$ are the points of discontinuity of $\Lambda_\gamma$ (which are necessarily
points in the spectrum of $\mathfrak P_\gamma$ by part~\eqref{itm:iterazioneindice2} of Theorem~\ref{thm:iterazioneindice}), $j=1,\ldots,K$,
setting:
\begin{equation}\label{eq:defbetaj}
\beta_j=\lim_{\theta\to0^+}\Lambda_\gamma\big(e^{2\pi i(\theta_j+\theta)}\big)\in\N\setminus\{0\},
\end{equation}
then the integral $\int_{\mathds S^1}\Lambda_\gamma\,\mathrm d\rho$ is given by the sum \eqref{eq:formulasomma}, where:
\[a_0=\beta_K,\quad a_1=\beta_K-\beta_1,\quad a_j=\beta_{j-1}-\beta_j,\ j=2,\ldots,K-1.\qedhere\]
\end{proof}
In order to apply equivariant Morse theory to the closed geodesic variational problem,
one needs a somewhat finer estimate on the growth of the index by iteration. More precisely,
it is needed a sort of \emph{uniform superlinear growth} for the sequence $\mu(\gamma^N)$ (see \cite[\S~1]{GroMey2}).
The result of Proposition~\ref{thm:stimacrescitalineare} can be improved as follows.
\begin{prop}\label{eq:stimamigliore}
Either $\mu(\gamma^N)$ is a constant sequence, or there exist constants $\alpha>0$ and $\beta\in\R$ such that:
\[\mu(\gamma^{N+s})-\mu(\gamma^N)\ge\alpha\cdot s+\beta\]
for all $N,s\in\N$.
\end{prop}
\begin{proof}
As above, let $e^{2\pi i\theta_j}$, be the discontinuity points of the map $\Lambda_\gamma$, where
$0\le\theta_1<\theta_2<\ldots<\theta_K<1$, set $\theta_{K+1}=\theta_1+1$, and define $\beta_j$ as in
\eqref{eq:defbetaj}. Note that, by \eqref{itm:iterazioneindice3} in Theorem~\ref{thm:iterazioneindice},
$\Lambda_\gamma\big(e^{2\pi i\theta_j}\big)\le\beta_j$ for all $j$.
If we denote by $\lfloor\cdot\rfloor$ the integer part function, the number of
$(N+s)$-th roots of unity that lie in the open arc $\{e^{2\pi i\theta}:\theta_j<\theta<\theta_{j+1}\}$
is at least $\lfloor (N+s)(\theta_{j+1}-\theta_j)\rfloor-1$. Similarly, the number of $N$-th roots of unity
that lie in the arc
$\{e^{2\pi i\theta}:\theta_j\le\theta<\theta_{j+1}\}$ is at most $\lfloor N(\theta_{j+1}-\theta_j)\rfloor+1$.
Thus, the following inequality holds:
\begin{multline}\label{eq:deslonga}
\mu(\gamma^{N+s})-\mu(\gamma^N)=\sum_{k=1}^{N+s}\Lambda_\gamma\big(e^{2\pi ik/(N+s)}\big)-\sum_{k=1}^{N}\Lambda_\gamma\big(e^{2\pi ik/N}\big)\\\ge
\sum_{j=1}^{K+1}\big(\lfloor (N+s)(\theta_{j+1}-\theta_j)\rfloor-1\big)\beta_j-\sum_{j=1}^{K+1}\big(\lfloor N(\theta_{j+1}-\theta_j)\rfloor+1\big)\beta_j\\
\ge\sum_{j=1}^{K+1}\big(\lfloor (N+s)(\theta_{j+1}-\theta_j)\rfloor-\lfloor N(\theta_{j+1}-\theta_j)\rfloor \big)\beta_j-2(K+1)\max\Lambda_\gamma\\
\ge\sum_{j=1}^{K+1}\big(\lfloor s(\theta_{j+1}-\theta_j)\rfloor-1\big)\beta_j-2(K+1)\max\Lambda_\gamma\\\ge \sum_{j=1}^{K+1}\big(\lfloor s(\theta_{j+1}-\theta_j)\rfloor\big)\beta_j
-3(K+1)\max\Lambda_\gamma.
\end{multline}
The assumption that $\mu(\gamma^N)$ is not a constant sequence implies the existence of at least one $\theta_0\in\left[0,1\right[$
such that $\Lambda_\gamma\big(e^{2\pi i\theta_0}\big)>0$; then $\theta_0\in\left[\theta_{j_0},\theta_{j_0+1}\right[$ for
some $j_0$, which implies $\beta_{j_0+1}-\beta_{j_0}>0$ and $\beta_{j_0}>0$. From \eqref{eq:deslonga} we therefore obtain:
\begin{multline*}
\mu(\gamma^{N+s})-\mu(\gamma^N)\ge \lfloor s(\theta_{j_0+1}-\theta_{j_0})\rfloor\cdot\beta_{j_0}
-3(K+1)\max\Lambda_\gamma\\\ge s(\theta_{j_0+1}-\theta_{j_0})\beta_{j_0}-\beta_{j_0}-3(K+1)\max\Lambda_\gamma.
\end{multline*}
This concludes the proof.
\end{proof}
\subsection{Hyperbolic geodesics}
A closed geodesic $\gamma:[0,1]\to M$ is said to be \emph{hyperbolic} if the linearized Poincar\'e map $\mathfrak P_\gamma$
has no eigenvalues on the unit circle. We say that a closed geodesic $\gamma$ is \emph{strongly hyperbolic} if, in addition,
$\epsilon_\gamma=0$. Observe that if $\gamma$ is (strongly) hyperbolic, then $\gamma^N$ is also (strongly) hyperbolic
for all $N\ge1$.
\begin{lem}
If $\gamma$ is  hyperbolic, then  $\mu(\gamma^N)=\epsilon_\gamma+N\cdot\mu_0(\gamma)$
for all $N\ge1$. If $\gamma$ is strongly hyperbolic, then $\mu(\gamma^N)=\mu_0(\gamma^N)=N\cdot\mu_0(\gamma)$.
\end{lem}
\begin{proof}
Immediate using  Theorem~\ref{thm:iterazioneindice};
here $\mu_0(\gamma)$ is the constant value of the function $\Lambda_\gamma$ on the unit circle.
\end{proof}

Let us assume that the Killing vector field $\mathcal K$ is complete, and let us denote
by $\psi_t:M\to M$ its flow, $t\in\R$, which consists of global isometries of $(M,g)$. We recall that two closed
geodesics $\gamma_i:[0,1]\to M$, $i=1,2$, are said to be \emph{geometrically distinct} if
there exists no $t\in\R$ such that $\psi_t\circ\gamma_1\big([0,1]\big)=\gamma_2\big([0,1]\big)$.
Let us denote by $\Lambda M$ the free loop space of $M$, which consists of all closed curves
$c:[0,1]\to M$ of Sobolev class $H^1$, endowed with the $H^1$-topology. Moreover, given a spacelike
hypersurface $S\subset M$, let ${\mathcal N}_S$ be the
subset of $\Lambda M$ consisting of those curves $c$ such that the quantity $g(\dot c,\mathcal K)$ is constant
on $[0,1]$, and with $c(0)\in S$. $\mathcal N_S$ is a smooth, closed, embedded submanifold of $\Lambda M$;
let us recall from \cite{BilMerPic} (see also \cite{CanFloSan}) the following result:
\begin{prop}\label{thm:minconncomp}
Let $(M,g)$ be a Lorentzian manifold endowed with a complete timelike Killing vector field
and a compact Cauchy surface $S$. Then, there exists a closed geodesic in every connected component
of $\Lambda M$; more precisely, the inclusion ${\mathcal N}_S\hookrightarrow\Lambda M$ is a homotopy equivalence,
and the geodesic functional $f(c)=\frac12\int_0^1g(\dot c,\dot c)\,\mathrm ds$
is bounded from below and has a minimum point in every connected component of $\mathcal N_S$, which is a geodesic
in $(M,g)$. The Morse index of a critical point $\gamma$ of $f$ in $\mathcal N_S$ equals $\mu(\gamma)$.\qed
\end{prop}
Recall that arc-connected components of $\Lambda M$ correspond to conjugacy classes of the fundamental
group $\pi_1(M)$; given one such  component $\Lambda_*$, we will call \emph{minimal} a closed
geodesic $\gamma$ in $\Lambda_*$, with $\gamma(0)\in S$, which is a minimum point for the restriction
of $f$ to the arc-connected component of $\mathcal N_S$ containing $\gamma$.
If $M$ is not simply connected, Proposition~\ref{thm:minconncomp} gives a multiplicity of  minimal
closed geodesics, however, there is no way of telling whether these geodesics are geometrically
distinct. Let us recall that there is a continuous action of the orthogonal group $\mathrm O(2)$ on the
free loop space $\Lambda M$, obtained from the action of $\mathrm O(2)$ on the parameter space
$\mathds S^1$. The geodesic functional $f$ is invariant by this action. Let us also recall
that the stabilizer of each point in $\Lambda M$ is a finite cyclic subgroup of $\mathrm{SO}(2)$, and
thus the critical $\mathrm O(2)$-orbits of $f$ are smooth submanifolds of $\Lambda M$ that are
diffeomorphic to two copies of the circle $\mathds S^1$. Using the flow of the Killing field
$\mathcal K$, one has also a free $\R$-action on $\Lambda M$
given by $\R\times\Lambda M\ni(t,\gamma)\mapsto\psi_t\circ\gamma\in\Lambda M$, and $f$ is invariant
by this action. The quotient $\Lambda M/\R$ can be identified in an obvious way with the submanifold
$\mathcal N_S$; since the actions of $\R$ and of $\mathrm O(2)$  commute, one can define a continuous
$\mathrm O(2)$-action on $\mathcal N_S$.

Inspired by a classical Riemannian result proved in \cite{BalThoZil}, we give the following:
\begin{prop}\label{thm:BTZ}
Under the hypotheses of Proposition~\ref{thm:minconncomp}, assume that there exists a non trivial element
$\mathfrak a$ in $\pi_1(M)$ satisfying the following:
\begin{itemize}
\item there exist integers $n\ne m$ such that the free homotopy classes generated by the conjugacy classes of $\mathfrak a^n$ and $\mathfrak a^m$ coincide;
\item for all $N\ge1$, every geodesic in the free homotopy class of $\mathfrak a^N$ is strongly hyperbolic.
\end{itemize}
Then, there are infinitely many \emph{geometrically distinct} closed geodesics in $(M,g)$.
\end{prop}
\begin{proof}
Let $\gamma$ be a minimal hyperbolic geodesic in the connected component of $\Lambda M$ determined by the
free homotopy class of $\mathfrak a$; then, $\gamma^n$ and $\gamma^m$ are freely homotopic,
and so are $\gamma^{nl}$ and $\gamma^{ml}$ for all $l\ge1$. Since $\gamma$ is minimal, then $\mu(\gamma)=0$,
and since $\gamma$ is hyperbolic, $\mu(\gamma^{nl})=\mu(\gamma^{ml})=0$ for all $l$. As proved in \cite{BilMerPic}, the geodesic
action functional is bounded from below and it satisfies the Palais--Smale condition on $\mathcal N_S$;
the (strong) hyperbolicity assumption implies that $f$ is an $\mathrm O(2)$-invariant Morse function
on each arc-connected component of $\mathcal N_S$ determined by the free homotopy class of some iterate of $\mathfrak a$.
Since $\gamma^{nl}$ and $\gamma^{ml}$ are in the same arc-connected component of $\mathcal N_S$
and they have index $0$, a classical result of equivariant Morse theory (strong Morse relations, see \cite{Cha, MawWil})
implies the existence of another
critical orbit $\mathrm O(2)c_l$, where $c_l\in\mathcal N_S$ is freely homotopic to $\gamma^{nl}$ and $\gamma^{ml}$, and whose Morse index
is equal to $1$. Observe that distinct critical $\mathrm O(2)$-orbits $\mathrm O(2)c_a$ and
$\mathrm O(2)c_b$ of $f$ in $\mathcal N_S$ determine geometrically
distinct closed geodesics if and only if $a$ and $b$ are not iterate one of the other.
By the strong hyperbolicity assumption, the  iterate $c_l^N$ has index equal to $N$ for all $N\ge1$;
in particular, the $c_l$'s are pairwise geometrically distinct, which concludes the proof.
\end{proof}
The question of existence and multiplicity of closed geodesics in stationary spacetimes seems more involved than in the Riemannian case, and there are few precedent results. We can cite for example \cite{Masiello2} for the existence of a closed geodesic and the recent paper \cite{BilMerPic} for a generalization of a Gromoll-Meyer type result. For the Riemannian case the main reference is \cite{Kli}.
\end{section}
\begin{section}{Final remarks, conjectures and open questions}

Hyperbolic closed geodesics that are singular are obviously strongly hyperbolic.
In view of Example~\ref{exa:casostatico}, in the case of static Lorentzian manifolds
Proposition~\ref{thm:BTZ} reproduces the central result in \cite{BalThoZil}.
On the other hand, the topological conditions on the fundamental group of
the manifold allows the authors of \cite{BalThoZil} to establish their infinitude results for a
\emph{generic} collection of Riemannian metrics on the given manifold.
This is based on a result due to Klingenberg and Takens \cite{KliTak}
that, given a compact
differential manifold $M$, the assumptions
of the Birkhoff--Lewis symplectic fixed point theorem holds for the Poincar\'e map
of every non hyperbolic closed geodesic for a $C^4$-generic set of Riemannian metrics
$\mathfrak g$ on $M$. No such result is known in Lorentzian geometry; even more,
it is not even known whether Lorentzian \emph{bumpy metrics} are generic.
Recall that a metric is bumpy if all its closed geodesics are nondegenerate;
in the case of stationary Lorentzian metrics, such definition clearly needs to be
adapted.

The genericity of Riemannian bumpy metrics on a compact manifold was proven
by Abraham in \cite{Abr}; a more recent elegant proof is given in  \cite{Whi}.
The central point in \cite{Whi} is that the Jacobi operator is strongly elliptic;
a similar property is satisfied by the differential operator obtained from the
second variation of the constrained variational problem in Proposition~\ref{thm:minconncomp}.
This suggests the conjecture that, also in the case of stationary Lorentzian manifolds
with a compact Cauchy surface, bumpy metrics are generic.

\smallskip

The recently developed theory of stationary Lorentzian closed geodesics and
their iteration (see also \cite{BilMerPic}) suggests that a number of classical Riemannian results
can be generalized to this context.
For instance, one cannot avoid mentioning a possible extension
to the stationary Lorentzian case of a result due to Bangert and Hinsgton \cite{BanHin}.
The authors' beautiful argument, based on Lusternik--Schnirelman theory, gives the existence
of infinitely many closed geodesics in compact Riemannian manifolds whose fundamental
group is infinite and abelian. We conjecture that the same result holds in the case of globally hyperbolic stationary Lorentzian
manifolds.
More results on the infinitude of closed Riemannian
based on the study of the homology generated by a tower of iterates
can be found in  \cite{BanKli}. Research in this direction for stationary Lorentzian manifolds
is being carried out, and it will be discussed in forthcoming papers.
\smallskip

Finally, we observe that a quite challenging task in the development of Morse theory for
closed Lorentzian geodesics would be removing the stationarity assumption. In this
case, one should deal with a truly strongly indefinite functional. Relations between
its critical points and the homological properties of the free loop space must then be
obtained by a more involved Morse theory based on a doubly infinite chain complex
determined by the dynamics of the gradient flow (see \cite{AbbMej} for the case of
geodesics between fixed endpoints). In this case, the notion
of Morse index has to be replaced by that of \emph{spectral flow} for a path
of Fredholm bilinear forms. We believe that the iteration results proven in this paper
generalize to spectral flows.

\end{section}

\end{document}